\documentclass[a4paper,12pt]{amsart}
\usepackage{amssymb,amsfonts,amsmath,amsthm}
\usepackage{amsrefs}
\usepackage{tikz}
\usepackage{doi}
\usepackage[margin=1in]{geometry}

\def\={\; = \;}
\def\leqs{\; \le \;}

\def\eqs{\; \equiv \;}
\def\:{\; \colon \;}
\def\ftt{\widehat}

\newcommand*\pmat[4]{\begin{pmatrix}#1&#2\\#3&#4\end{pmatrix}}
\newcommand*\smat[4]{\begin{smallmatrix}#1&#2\\#3&#4\end{smallmatrix}}

\newcommand{\Mod}[1]{\ (\text{mod}\ #1)}
\newcommand{\ol}[1]{\overline{#1}}
\newcommand{\es}[1]{\; #1 \;}

\newcommand{\eps}{\varepsilon}
\newcommand{\sm}{\smallsetminus}

\newcommand{\RR}{{\mathbb{R}}}
\newcommand{\ZZ}{{\mathbb{Z}}}
\newcommand{\CC}{{\mathbb{C}}}
\newcommand{\QQ}{{\mathbb{Q}}}
\newcommand{\NN}{{\mathbb{N}}}
\newcommand{\HH}{{\mathfrak{H}}}

\newcommand{\thU}{\Theta_{3}}
\newcommand{\thV}{\Theta_{2}}
\newcommand{\thW}{\Theta_{4}}

\newcommand{\jth}{J}
\newcommand{\rds}{\mathfrak{s}}

\newcommand{\dx}{dx}
\newcommand{\dy}{dy}
\newcommand{\dz}{dz}
\newcommand{\dt}{dt}
\newcommand{\SL}{\operatorname{SL}}
\newcommand{\PSL}{\operatorname{PSL}}
\renewcommand{\Re}{\mathrm{Re}}
\renewcommand{\Im}{\mathrm{Im}}

\theoremstyle{plain}
\newtheorem{theorem}{Theorem}
\newtheorem{lemma}{Lemma}

\newtheorem{proposition}{Proposition}
\newtheorem{corollary}{Corollary}
\newtheorem{question}{Question}

\title{Fourier interpolation on the real line}
\date{}
\author[Radchenko]{Danylo Radchenko}
\address{The Abdus Salam International Centre for Theoretical Physics,
	Str.\ Costiera 11, 34151 Trieste, Italy}
\email{danradchenko@gmail.com}

\author[Viazovska]{Maryna Viazovska}
\address{\'{E}cole Polytechnique F\'{e}d\'{e}rale de Lausanne, 1015 Lausanne, Switzerland}
\email{viazovska@gmail.com}

\begin{document}

\begin{abstract}
	In this paper we
	construct an explicit interpolation formula for Schwartz functions
    on the real line.
	The formula expresses the value of a function at any given point in
    terms of the values of the function and its Fourier transform
	on the set $\{0, \pm\sqrt{1}, \pm\sqrt{2}, \pm\sqrt{3},\dots\}$. The functions in the interpolating basis are constructed in a closed form  as an integral transform of weakly holomorphic modular forms for the theta subgroup of the modular group.
\end{abstract}
\maketitle
\section{Introduction}
Let $f\colon\RR\to\RR$ be an integrable function and let $\ftt{f}$ be the
Fourier transform of $f$:
	\[\ftt{f}(\xi) \= \int_{-\infty}^{\infty}f(x)
    e^{-2\pi i\xi x}\dx.\]
The classical Whittaker-Shannon interpolation
formula (see~\cite{W},~\cite{Sha}) states that if the Fourier
transform $\ftt{f}$ is supported in $[-w/2,w/2]$, then
	\begin{equation}
    \label{eq:pw_interpolation}
    f(x) \= \sum_{n\in\ZZ}f(n/w)\, {\rm sinc}(wx-n),
    \end{equation}
where ${\rm sinc}(x) = \sin(\pi x)/(\pi x)$ is the cardinal sine function.
In other words, the functions~$s_n(x)={\rm sinc}(wx-n)$ form an
interpolation basis on the set $\frac{1}{w}\ZZ$ for the space of
functions whose Fourier transform is supported in $[-w/2,w/2]$
(the so-called Paley-Wiener space $PW_w$).
For a nice overview of history of the Whittaker-Shannon formula, its
generalizations and other related results, see~\cite{Hig}.

Note that it is not possible to apply the Whittaker-Shannon formula directly
to functions whose Fourier transform $\ftt{f}$ has unbounded support, say,
to $f(x)=\exp(-\pi x^2)$.
The main goal of this paper is to prove an interpolation formula that can
be applied to arbitrary Schwartz functions on the real line.
\begin{theorem}\label{thm:summation}
	There exists a collection of even Schwartz functions $a_n\colon\RR\to\RR$ with
	the property that for any even Schwartz function $f\colon\RR\to\RR$ and any $x \in \RR$ we have
	\begin{equation} \label{eq:summation}
	f(x) \= \sum_{n=0}^{\infty} a_{n}(x)f(\sqrt{n})+\sum_{n=0}^{\infty} \widehat{a}_{n}(x)\widehat{f}(\sqrt{n}),
	\end{equation}
	where the right-hand side converges absolutely.
\end{theorem}

As immediate corollary of Theorem~\ref{thm:summation},
we get the following.
\begin{corollary}
	Let $f\colon\RR\to\RR$ be an even Schwartz function that satisfies
	\begin{equation*}
		f(\sqrt{n}) \= \ftt{f}(\sqrt{n}) \= 0, \quad n\in\ZZ_{\ge 0}.
	\end{equation*}
	Then $f$ vanishes identically.
\end{corollary}
Denote by $\rds$ the vector space of all rapidly decaying sequences of
real numbers, i.e., sequences $(x_n)_{n\ge0}$ such that
for all $k>0$ we have $n^kx_n \to 0, n\to\infty$. If we denote
by~$\mathcal{S}_{even}$ the space of even Schwartz functions on $\RR$
(see Section~\ref{sec:funcan} for a formal definition),
then there is a well-defined
map~$\Psi\colon\mathcal{S}_{even}\to\rds\oplus \rds$ given by
	\[\Psi(f) \= (f(\sqrt{n}))_{n\ge 0} \oplus (\ftt{f}(\sqrt{n}))_{n\ge 0}.\]
Together with Theorem~\ref{thm:summation} the following result gives a complete
description of what values an even Schwartz function and its Fourier transform
can take at $\pm\sqrt{n}$ for $n\ge0$.
\begin{theorem} \label{thm:isomorphism}
	The map $\Psi$ is an isomorphism of the space of even Schwartz functions
	onto the vector space $\ker L \subset \rds\oplus\rds$, where $L\colon\rds\oplus\rds\to\RR$
	is the linear functional
	\[L((x_n)_{n\ge0}, (y_n)_{n\ge0}) \= \sum_{n\in\ZZ} x_{n^2} \es{-} \sum_{n\in\ZZ} y_{n^2}.\]
\end{theorem}
In the proof of Theorem~\ref{thm:summation} we will give an
explicit construction of the interpolating basis $\{a_n(x)\}_{n\ge0}$.
For instance, the Fourier invariant part of $a_n$ will be given~by
	\[a_n(x)+\ftt{a_n}(x)\=\int_{-1}^{1} g_n(z)\,e^{i\pi x^2 z}\dz,\]
where $g_n$ is a certain weakly holomorphic modular form of
weight~$3/2$, and the integral is over a semicircle in the upper
half-plane.
The anti-invariant part $a_n(x)-\ftt{a_n}(x)$ will be
defined by a similar expression. For an explicit example, we
define $a_0(x)$ by
	\[a_0(x) \= \frac{1}{4}\int_{-1}^{1}\theta^3(z)\,e^{i\pi x^2 z}\dz,\]
where $\theta(z)$ is the classical theta series
	\begin{equation}\label{eq:theta}
	\theta(z)\=\sum_{n\in\ZZ}e^{i\pi n^2 z}.
	\end{equation}
The modular transformation property of $g_n$ is chosen in such a way
that it complements the action of the Fourier transform on Gaussian
functions:
	\[\ftt{e_{z}}(\xi) \= \frac{1}{\sqrt{-iz}}\,e_{-1/z}(\xi),\]
where $e_{z}(x) = e^{i\pi z x^2}$, and the square root is chosen to be
positive when $z$ lies on the imaginary axis (this comment also applies whenever expression $(-iz)^{\alpha}$ occurs throughout the paper; note that $z$ belongs to the upper half-plane).
For instance, using the identity
	\begin{equation*}
	\theta\Big(-\frac{1}{z}\Big) \= \sqrt{-iz}\,\theta(z)
	\end{equation*}
and applying the change of variable $z\mapsto-1/z$ in the integral
that defines $a_0(x)$ we see that $\ftt{a_0}=a_0$.
The general definition of $a_n$ needs some preparation and will be given
in Section~\ref{sec:basis}. The plots of the first three functions are
shown in Figure~\ref{figure:plots}.

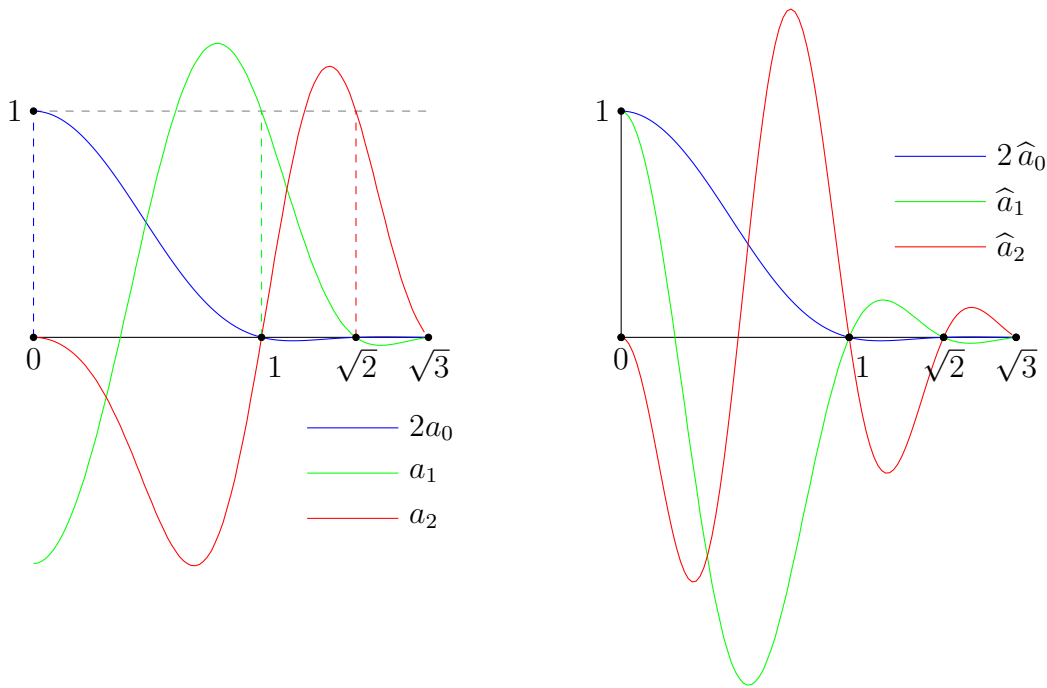
\begin{figure}
	\begin{tikzpicture}
	\begin{scope}[scale=0.6]
	\draw (0,0)--(8.6603,0) ;
	\draw [gray, dashed]
	(0,5)--(8.6603,5);
	
	\draw [blue]
	(0,5.000)--(0.08660,4.996)--(0.1732,4.985)--(0.2598,4.966)--(0.3464,4.940)--(0.4330,4.906)--(0.5196,4.865)--(0.6062,4.817)--(0.6928,4.762)--(0.7794,4.700)--(0.8660,4.632)--(0.9526,4.557)--(1.039,4.476)--(1.126,4.389)--(1.212,4.297)--(1.299,4.199)--(1.386,4.096)--(1.472,3.989)--(1.559,3.877)--(1.645,3.762)--(1.732,3.642)--(1.819,3.520)--(1.905,3.395)--(1.992,3.267)--(2.078,3.137)--(2.165,3.006)--(2.252,2.873)--(2.338,2.740)--(2.425,2.606)--(2.511,2.473)--(2.598,2.340)--(2.685,2.208)--(2.771,2.077)--(2.858,1.948)--(2.944,1.821)--(3.031,1.696)--(3.118,1.574)--(3.204,1.456)--(3.291,1.340)--(3.377,1.229)--(3.464,1.121)--(3.551,1.017)--(3.637,0.9181)--(3.724,0.8236)--(3.811,0.7338)--(3.897,0.6489)--(3.984,0.5689)--(4.070,0.4940)--(4.157,0.4242)--(4.244,0.3595)--(4.330,0.2999)--(4.417,0.2453)--(4.503,0.1956)--(4.590,0.1508)--(4.677,0.1108)--(4.763,0.07526)--(4.850,0.04415)--(4.936,0.01723)--(5.023,-0.005703)--(5.110,-0.02489)--(5.196,-0.04057)--(5.283,-0.05301)--(5.369,-0.06246)--(5.456,-0.06920)--(5.543,-0.07349)--(5.629,-0.07562)--(5.716,-0.07585)--(5.802,-0.07444)--(5.889,-0.07166)--(5.976,-0.06776)--(6.062,-0.06296)--(6.149,-0.05749)--(6.235,-0.05156)--(6.322,-0.04535)--(6.409,-0.03904)--(6.495,-0.03278)--(6.582,-0.02670)--(6.668,-0.02091)--(6.755,-0.01550)--(6.842,-0.01055)--(6.928,-0.006115)--(7.015,-0.002224)--(7.101,0.001100)--(7.188,0.003854)--(7.275,0.006048)--(7.361,0.007705)--(7.448,0.008857)--(7.534,0.009547)--(7.621,0.009821)--(7.708,0.009734)--(7.794,0.009342)--(7.881,0.008700)--(7.967,0.007867)--(8.054,0.006897)--(8.141,0.005841)--(8.227,0.004749)--(8.314,0.003661)--(8.400,0.002615)--(8.487,0.001642)--(8.574,0.0007639);
	\draw [green]
	(0,-5.000)--(0.0866,-4.988)--(0.173,-4.952)--(0.260,-4.892)--(0.346,-4.808)--(0.433,-4.701)--(0.520,-4.571)--(0.606,-4.418)--(0.693,-4.243)--(0.779,-4.047)--(0.866,-3.830)--(0.953,-3.593)--(1.04,-3.337)--(1.13,-3.063)--(1.21,-2.771)--(1.30,-2.464)--(1.39,-2.142)--(1.47,-1.806)--(1.56,-1.458)--(1.65,-1.100)--(1.73,-0.7316)--(1.82,-0.3560)--(1.91,0.02563)--(1.99,0.4117)--(2.08,0.8004)--(2.17,1.190)--(2.25,1.579)--(2.34,1.964)--(2.42,2.346)--(2.51,2.720)--(2.60,3.087)--(2.68,3.443)--(2.77,3.787)--(2.86,4.117)--(2.94,4.431)--(3.03,4.729)--(3.12,5.007)--(3.20,5.265)--(3.29,5.502)--(3.38,5.716)--(3.46,5.905)--(3.55,6.069)--(3.64,6.208)--(3.72,6.320)--(3.81,6.405)--(3.90,6.462)--(3.98,6.492)--(4.07,6.495)--(4.16,6.470)--(4.24,6.419)--(4.33,6.342)--(4.42,6.240)--(4.50,6.114)--(4.59,5.965)--(4.68,5.795)--(4.76,5.605)--(4.85,5.397)--(4.94,5.174)--(5.02,4.936)--(5.11,4.686)--(5.20,4.426)--(5.28,4.159)--(5.37,3.886)--(5.46,3.610)--(5.54,3.333)--(5.63,3.057)--(5.72,2.784)--(5.80,2.516)--(5.89,2.256)--(5.98,2.004)--(6.06,1.762)--(6.15,1.532)--(6.24,1.315)--(6.32,1.112)--(6.41,0.9242)--(6.50,0.7514)--(6.58,0.5943)--(6.67,0.4531)--(6.75,0.3278)--(6.84,0.2182)--(6.93,0.1238)--(7.01,0.04420)--(7.10,-0.02147)--(7.19,-0.07401)--(7.27,-0.1144)--(7.36,-0.1437)--(7.45,-0.1631)--(7.53,-0.1736)--(7.62,-0.1766)--(7.71,-0.1732)--(7.79,-0.1646)--(7.88,-0.1519)--(7.97,-0.1361)--(8.05,-0.1184)--(8.14,-0.09951)--(8.23,-0.08033)--(8.31,-0.06152)--(8.40,-0.04367)--(8.49,-0.02725)--(8.57,-0.01261);
	\draw [red]
	(0,0.0000999708)--(0.0866,-0.00308556)--(0.173,-0.0126792)--(0.260,-0.0287908)--(0.346,-0.0515992)--(0.433,-0.0813471)--(0.520,-0.118333)--(0.606,-0.162904)--(0.693,-0.215445)--(0.779,-0.276365)--(0.866,-0.346087)--(0.953,-0.425032)--(1.04,-0.513601)--(1.13,-0.612157)--(1.21,-0.721011)--(1.30,-0.840403)--(1.39,-0.970485)--(1.47,-1.1113)--(1.56,-1.26279)--(1.65,-1.42472)--(1.73,-1.59674)--(1.82,-1.77832)--(1.91,-1.96876)--(1.99,-2.16715)--(2.08,-2.37243)--(2.17,-2.5833)--(2.25,-2.79832)--(2.34,-3.0158)--(2.42,-3.2339)--(2.51,-3.45058)--(2.60,-3.66367)--(2.68,-3.87081)--(2.77,-4.06954)--(2.86,-4.25729)--(2.94,-4.43142)--(3.03,-4.58923)--(3.12,-4.72803)--(3.20,-4.84514)--(3.29,-4.93794)--(3.38,-5.00393)--(3.46,-5.04075)--(3.55,-5.04622)--(3.64,-5.01839)--(3.72,-4.95559)--(3.81,-4.85645)--(3.90,-4.71995)--(3.98,-4.54545)--(4.07,-4.33272)--(4.16,-4.08197)--(4.24,-3.79387)--(4.33,-3.46954)--(4.42,-3.11061)--(4.50,-2.71917)--(4.59,-2.29778)--(4.68,-1.84947)--(4.76,-1.3777)--(4.85,-0.886315)--(4.94,-0.379557)--(5.02,0.138037)--(5.11,0.661662)--(5.20,1.18631)--(5.28,1.70683)--(5.37,2.21802)--(5.46,2.71467)--(5.54,3.19167)--(5.63,3.64408)--(5.72,4.06717)--(5.80,4.45657)--(5.89,4.8083)--(5.98,5.1188)--(6.06,5.38509)--(6.15,5.60473)--(6.24,5.77591)--(6.32,5.89749)--(6.41,5.96899)--(6.50,5.99064)--(6.58,5.96335)--(6.67,5.8887)--(6.75,5.76893)--(6.84,5.6069)--(6.93,5.40598)--(7.01,5.17008)--(7.10,4.9035)--(7.19,4.6109)--(7.27,4.29718)--(7.36,3.96739)--(7.45,3.62667)--(7.53,3.28009)--(7.62,2.93262)--(7.71,2.589)--(7.79,2.25367)--(7.88,1.9307)--(7.97,1.62371)--(8.05,1.33582)--(8.14,1.06962)--(8.23,0.827121)--(8.31,0.609746)--(8.40,0.418346)--(8.49,0.253198)--(8.57,0.114038);
	\draw [blue, dashed]
	(0,0)--(0,5);
	\draw [green, dashed]
	(5,0)--(5,5);
	\draw [red, dashed]
	(7.0711,0)--(7.0711,5);
	\filldraw (0,5) circle (2pt) node[left] {$1$};
	\filldraw (0,0) circle (2pt) node[align=center,below] {$0$};
	\filldraw (5,0) circle (2pt) node[align=center,below,xshift=5,yshift=-3] {$1$};
	\filldraw (7.0711,0) circle (2pt) node[align=center,below] {$\sqrt{2}$};
	\filldraw (8.6603,0) circle (2pt) node[align=center,below] {$\sqrt{3}$};
	\filldraw (8,-2) node[right] {$2a_0$};
	\draw [blue]
	(6,-2)--(8,-2);
	\filldraw (8,-3) node[right] {$a_1$};
	\draw [green]
	(6,-3)--(8,-3);
	\filldraw (8,-4) node[right] {$a_2$};
	\draw [red]
	(6,-4)--(8,-4);
	\end{scope}
	
	\begin{scope}[xshift=220,scale=0.6]
	\draw (0,0)--(8.6603,0) ;
	\draw (0,0)--(0,5) ;
	
	\draw [blue]
	(0,5.000)--(0.08660,4.996)--(0.1732,4.985)--(0.2598,4.966)--(0.3464,4.940)--(0.4330,4.906)--(0.5196,4.865)--(0.6062,4.817)--(0.6928,4.762)--(0.7794,4.700)--(0.8660,4.632)--(0.9526,4.557)--(1.039,4.476)--(1.126,4.389)--(1.212,4.297)--(1.299,4.199)--(1.386,4.096)--(1.472,3.989)--(1.559,3.877)--(1.645,3.762)--(1.732,3.642)--(1.819,3.520)--(1.905,3.395)--(1.992,3.267)--(2.078,3.137)--(2.165,3.006)--(2.252,2.873)--(2.338,2.740)--(2.425,2.606)--(2.511,2.473)--(2.598,2.340)--(2.685,2.208)--(2.771,2.077)--(2.858,1.948)--(2.944,1.821)--(3.031,1.696)--(3.118,1.574)--(3.204,1.456)--(3.291,1.340)--(3.377,1.229)--(3.464,1.121)--(3.551,1.017)--(3.637,0.9181)--(3.724,0.8236)--(3.811,0.7338)--(3.897,0.6489)--(3.984,0.5689)--(4.070,0.4940)--(4.157,0.4242)--(4.244,0.3595)--(4.330,0.2999)--(4.417,0.2453)--(4.503,0.1956)--(4.590,0.1508)--(4.677,0.1108)--(4.763,0.07526)--(4.850,0.04415)--(4.936,0.01723)--(5.023,-0.005703)--(5.110,-0.02489)--(5.196,-0.04057)--(5.283,-0.05301)--(5.369,-0.06246)--(5.456,-0.06920)--(5.543,-0.07349)--(5.629,-0.07562)--(5.716,-0.07585)--(5.802,-0.07444)--(5.889,-0.07166)--(5.976,-0.06776)--(6.062,-0.06296)--(6.149,-0.05749)--(6.235,-0.05156)--(6.322,-0.04535)--(6.409,-0.03904)--(6.495,-0.03278)--(6.582,-0.02670)--(6.668,-0.02091)--(6.755,-0.01550)--(6.842,-0.01055)--(6.928,-0.006115)--(7.015,-0.002224)--(7.101,0.001100)--(7.188,0.003854)--(7.275,0.006048)--(7.361,0.007705)--(7.448,0.008857)--(7.534,0.009547)--(7.621,0.009821)--(7.708,0.009734)--(7.794,0.009342)--(7.881,0.008700)--(7.967,0.007867)--(8.054,0.006897)--(8.141,0.005841)--(8.227,0.004749)--(8.314,0.003661)--(8.400,0.002615)--(8.487,0.001642)--(8.574,0.0007639);
	\draw [green]
	(0,5.000)--(0.0866,4.968)--(0.173,4.873)--(0.260,4.716)--(0.346,4.499)--(0.433,4.223)--(0.520,3.892)--(0.606,3.510)--(0.693,3.079)--(0.779,2.605)--(0.866,2.093)--(0.953,1.547)--(1.04,0.9736)--(1.13,0.3786)--(1.21,-0.2321)--(1.30,-0.8523)--(1.39,-1.476)--(1.47,-2.097)--(1.56,-2.709)--(1.65,-3.306)--(1.73,-3.883)--(1.82,-4.433)--(1.91,-4.953)--(1.99,-5.438)--(2.08,-5.882)--(2.17,-6.284)--(2.25,-6.639)--(2.34,-6.944)--(2.42,-7.198)--(2.51,-7.400)--(2.60,-7.548)--(2.68,-7.643)--(2.77,-7.684)--(2.86,-7.672)--(2.94,-7.610)--(3.03,-7.498)--(3.12,-7.340)--(3.20,-7.138)--(3.29,-6.897)--(3.38,-6.618)--(3.46,-6.307)--(3.55,-5.968)--(3.64,-5.604)--(3.72,-5.222)--(3.81,-4.824)--(3.90,-4.416)--(3.98,-4.002)--(4.07,-3.588)--(4.16,-3.176)--(4.24,-2.771)--(4.33,-2.376)--(4.42,-1.996)--(4.50,-1.633)--(4.59,-1.291)--(4.68,-0.9703)--(4.76,-0.6743)--(4.85,-0.4042)--(4.94,-0.1611)--(5.02,0.05437)--(5.11,0.2419)--(5.20,0.4016)--(5.28,0.5339)--(5.37,0.6398)--(5.46,0.7204)--(5.54,0.7770)--(5.63,0.8114)--(5.72,0.8256)--(5.80,0.8214)--(5.89,0.8012)--(5.98,0.7671)--(6.06,0.7214)--(6.15,0.6664)--(6.24,0.6043)--(6.32,0.5373)--(6.41,0.4672)--(6.50,0.3961)--(6.58,0.3257)--(6.67,0.2573)--(6.75,0.1925)--(6.84,0.1321)--(6.93,0.07716)--(7.01,0.02828)--(7.10,-0.01408)--(7.19,-0.04970)--(7.27,-0.07851)--(7.36,-0.1007)--(7.45,-0.1164)--(7.53,-0.1262)--(7.62,-0.1306)--(7.71,-0.1301)--(7.79,-0.1255)--(7.88,-0.1175)--(7.97,-0.1067)--(8.05,-0.09399)--(8.14,-0.07995)--(8.23,-0.06527)--(8.31,-0.05052)--(8.40,-0.03622)--(8.49,-0.02282)--(8.57,-0.01066);
	\draw [red]
	(0,0.0000100292)--(0.08660,-0.0377607)--(0.1732,-0.150067)--(0.2598,-0.333918)--(0.3464,-0.584403)--(0.4330,-0.89481)--(0.5196,-1.25678)--(0.6062,-1.66052)--(0.6928,-2.09501)--(0.7794,-2.54828)--(0.8660,-3.00767)--(0.9526,-3.46019)--(1.039,-3.89274)--(1.126,-4.29249)--(1.212,-4.64714)--(1.299,-4.94526)--(1.386,-5.17652)--(1.472,-5.33195)--(1.559,-5.40416)--(1.645,-5.38756)--(1.732,-5.27841)--(1.819,-5.07502)--(1.905,-4.77772)--(1.992,-4.38892)--(2.078,-3.91304)--(2.165,-3.35643)--(2.252,-2.72724)--(2.338,-2.03525)--(2.425,-1.29166)--(2.511,-0.508857)--(2.598,0.299874)--(2.685,1.12061)--(2.771,1.93908)--(2.858,2.74102)--(2.944,3.51241)--(3.031,4.23983)--(3.118,4.91069)--(3.204,5.51354)--(3.291,6.03828)--(3.377,6.47636)--(3.464,6.82095)--(3.551,7.06707)--(3.637,7.2117)--(3.724,7.25375)--(3.811,7.19415)--(3.897,7.03575)--(3.984,6.78323)--(4.070,6.44298)--(4.157,6.02298)--(4.244,5.53252)--(4.330,4.98205)--(4.417,4.3829)--(4.503,3.74705)--(4.590,3.08683)--(4.677,2.41467)--(4.763,1.74284)--(4.850,1.08317)--(4.936,0.446836)--(5.023,-0.15588)--(5.110,-0.715767)--(5.196,-1.22485)--(5.283,-1.67655)--(5.369,-2.06571)--(5.456,-2.38872)--(5.543,-2.64352)--(5.629,-2.82957)--(5.716,-2.94778)--(5.802,-3.00047)--(5.889,-2.99124)--(5.976,-2.92478)--(6.062,-2.80679)--(6.149,-2.64373)--(6.235,-2.44266)--(6.322,-2.21103)--(6.409,-1.9565)--(6.495,-1.6867)--(6.582,-1.40909)--(6.668,-1.13076)--(6.755,-0.858295)--(6.842,-0.597617)--(6.928,-0.353892)--(7.015,-0.131445)--(7.101,0.0663028)--(7.188,0.236858)--(7.275,0.378656)--(7.361,0.491038)--(7.448,0.574203)--(7.534,0.629135)--(7.621,0.657514)--(7.708,0.661619)--(7.794,0.644207)--(7.881,0.608396)--(7.967,0.557541)--(8.054,0.49511)--(8.141,0.424568)--(8.227,0.349269)--(8.314,0.272359)--(8.400,0.196691)--(8.487,0.124759)--(8.574,0.0586501);
	\filldraw (0,0) circle (2pt);
	\filldraw (5,0) circle (2pt);
	\filldraw (7.0711,0) circle (2pt);
	\filldraw (8.6603,0) circle (2pt);
	\filldraw (0,5) circle (2pt) node[left] {$1$};
	\filldraw (0,0) circle (2pt) node[align=center, below] {$0$};
	\filldraw (5,0) circle (2pt) node[align=center,below,xshift=5,yshift=-3] {$1$};
	\filldraw (7.0711,0) circle (2pt) node[align=center, below] {$\sqrt{2}$};
	\filldraw (8.6603,0) circle (2pt) node[align=center, below] {$\sqrt{3}$};
	\filldraw (8,4) node[right] {$2\,\widehat{a}_0$};
	\draw [blue]
	(6,4)--(8,4);
	\filldraw (8,3) node[right] {$\widehat{a}_1$};
	\draw [green]
	(6,3)--(8,3);
	\filldraw (8,2) node[right] {$\widehat{a}_2$};
	\draw [red]
	(6,2)--(8,2);
	\end{scope}
	\end{tikzpicture}
		\caption{Plots of $a_n$ and $\ftt{a_n}$ for $n=0,1,2$.}
		\label{figure:plots}
	\end{figure}

An analogue of Theorem~\ref{thm:summation} holds also for odd Schwartz
functions, but we postpone its formulation until Section~\ref{sec:odd}.
It is possible to combine the two results into a general interpolation
theorem, but it is more convenient to work with the two cases separately.

\textbf{Remark.} Another way to interpret equation~\eqref{eq:summation} is to
think of it as a ``deformation'' of the classical Poisson summation formula
\begin{equation}\label{eq:poisson}
	\sum_{n\in\ZZ} f(n) \= \sum_{n\in\ZZ} \ftt{f}(n),\end{equation}
which will be a special case of~\eqref{eq:summation} for $x=0$
(more precisely, $-a_{n^2}(0) = \ftt{a_{n^2}}(0) = 1$ for~$n\ge 1$,
$a_0(0) = \ftt{a_0}(0) = 1/2$, and all other values are zero).
Note also that equation~\eqref{eq:summation} gives a continuous family
of measures $\mu_x$ such that $\mu_x$ is a tempered distribution,
and both $\mu_x$ and $\ftt{\mu_x}$ have locally finite support. Such
measures are called crystalline measures, for general discussion
and some interesting examples see~\cite{Me},~\cite{Guin}.

Our general approach fits into the framework of Eichler cohomology
(see~\cite{Eich}; some relevant results can also be found
in~\cite{Kn1} and~\cite{Kn2}) but for the most part we avoid using its
general results and terminology. In our case we prefer to obtain
explicit estimates by direct methods, and this also allows us to keep
the proofs relatively self-contained.

Let us also remark that functions with properties similar to that
of $a_n$ have recently been used in~\cite{Vi} and~\cite{CKMRV} to solve
the sphere packing problem in dimensions~8 and~24. The functions
constructed there, motivated by the Cohn-Elkies optimization problem
~\cite{CE}, were also solutions to a very special case of an
interpolation problem closely related to~\eqref{eq:summation} that
also involved the values of the first derivative. Similarly, in the
Paley-Wiener space, an analogue of~\eqref{eq:pw_interpolation}
for second-order interpolation (i.e., interpolation of values of the
function and the values of its first derivative) plays important role
in optimization problems of Beurling and Selberg, see~\cite{Va}.

The paper is organized as follows. In Section~\ref{sec:background}
we recall some known facts about modular forms for the theta
group~$\Gamma_{\theta}$. In Section~\ref{sec:mf} we compute an
explicit basis of a certain space of weakly holomorphic modular
forms of weight~$3/2$ for the group~$\Gamma_{\theta}$.
Then, in Section~\ref{sec:basis} we use these modular forms to
construct an interpolation basis for the even Schwartz functions
and prove some of its properties. In the next section we prove an estimate
on the growth of this sequence functions; this is by far the most
technical part of the paper.
In Section~\ref{sec:funcan} we prove the main result for even
functions, and in Section~\ref{sec:odd} we define the interpolation basis and formulate corresponding statements for the odd functions.

\subsection*{Acknowledgements.}
The authors would like to thank Max Planck Institute for Mathematics, Bonn for hospitality and support while this paper was being written.
The first named author would like to thank The Absus Salam International Centre for Theoretical Physics, Trieste for financial support.
The second named author would also like to thank the Berlin Mathematical School for financial support and excellent research environment.

The authors are grateful to Andrew Bakan, Andriy Bondarenko,
Emanuel Carneiro, Yves Meyer, Don Zagier, and the anonymous referee
for many helpful remarks and comments. The second named author is grateful to Andriy Bondarenko for sharing his conjecture about the existence of the interpolation formula.

\section{The theta group}
\label{sec:background}
In this section we set up notation and collect facts about the theta
group and related modular forms.
Most of the material from this section can be found, in much greater
detail, in~\cite{Mu}. For a motivated general introduction to the
theory of modular forms, see~\cite{Za2}.

\subsection{Upper half-plane and the action of $\SL_{2}(\RR)$}
Denote by~$\HH$ the complex upper half-plane
$\{z\in\CC \: \Im(z) > 0\}$. The group $\SL_2(\RR)$ of $2\times 2$
matrices with real coefficients and determinant~$1$ acts on the upper
half-plane on the left by Moebius transformations
	\[\gamma z \= \frac{az+b}{cz+d},\quad \gamma =
    \pmat abcd\in \SL_{2}(\RR).\]
The kernel of this action coincides with the center $\{\pm I\}$
of $\SL_{2}(\RR)$ and thus we can work with the action
of $\PSL_2(\RR) = \SL_2(\RR)/\{\pm I\}$ instead.

We will use special notation for the following elements
of $\SL_2(\ZZ)$ (or, by abuse of notation, of $\PSL_2(\ZZ)$):
	\begin{equation*}
	I \= \pmat 1001,\quad T \= \pmat 1101,\quad S \= \pmat 0{-1}10.
	\end{equation*}

Recall that $\Gamma(2) \subset \SL_2(\ZZ)$ is defined as
	\begin{equation*}
	\Gamma(2) \= \bigg\{A \in \SL_2(\ZZ)
	\ \bigg| \ A \eqs \pmat 1001 \Mod{2}\bigg\},
	\end{equation*}
and the theta group $\Gamma_{\theta}$ is the subgroup
of $\SL_2(\ZZ)$ generated by~$S$ and~$T^2$, or, equivalently,
	\begin{equation*}
		\Gamma_{\theta} \= \bigg\{A \in \SL_2(\ZZ)
		\ \bigg| \ A \eqs \pmat 1001\ \mbox{or}\ \pmat 0110 \Mod{2}\bigg\}.
	\end{equation*}
Note the obvious inclusions $\SL_{2}(\ZZ) \supset \Gamma_{\theta} \supset \Gamma(2)$.
The group $\Gamma(2)$ has three cusps $0$, $1$, and $\infty$, while the group
$\Gamma_{\theta}$ has only two cusps: $1$ and $\infty$. The standard
fundamental domain for the theta group is
(see Figure~\ref{figure:fundamental_domain})
	\begin{equation} \label{eq:funddom}
	\mathcal{D} \= \{\tau\in\HH \: |\tau|>1,\, \Re(\tau) \in (-1,1)\}.
	\end{equation}
	
\begin{figure}
    \centering
    \begin{tikzpicture}
    \definecolor{cv0}{rgb}{0.95,0.95,0.95}
    \begin{scope}[scale=1]
    \clip(-7,-0.5) rectangle (7,6);
    \draw (-3,0)node[below] {$-1$};
    \draw (3,0) node[below] {$1$};
    \draw[lightgray] (-10,0) -- (10,0);
    \draw[lightgray] (9,0) --  (9,5.5);
    \draw[lightgray] (-9,0)  --  (-9,5.5);

    \foreach\i in {-6,0,6} {
    \begin{scope}[xshift=\i cm]
    \draw[lightgray] (-3,0) arc (180:0:1) arc (180:0:0.5) arc    (0:180:1.5);
    \draw[lightgray] (3,0) arc (0:180:1) arc (0:180:0.5) arc    (180:0:1.5);
    \draw[lightgray] (-3,0) arc (180:0:0.75) arc (180:0:0.25) arc (180:0:0.2) arc (180:0:0.3) arc (180:0:0.3) arc (180:0:0.2) arc (180:0:0.25) arc (180:0:0.75);
    \end{scope}
    }

    \fill[color=cv0] (-3,0) arc (180:0:3) -- (3,5.5) -- (-3,5.5);
    \draw[lightgray] (3,0)  --  ( 3,5.5);
    \draw[lightgray] (-3,0)  -- (-3,5.5);
    \draw[lightgray] (3,0) arc (180:0:3);
    \draw[lightgray] (-3,0) arc (0:180:3);
    \draw[lightgray] (3,0) arc (0:180:3);
    \draw (0,4.5) node[above]{$\mathcal{D}$};
    \end{scope}
    \end{tikzpicture}
    \caption{Fundamental domain for $\Gamma_{\theta}$.}
    \label{figure:fundamental_domain}
\end{figure}

Finally, we are going to use the ``$\theta$-automorphy factor'' on the group $\Gamma_{\theta}$,
which we define for all $z\in\HH$ and $\gamma\in\Gamma_{\theta}$ by
	\begin{equation}
	j_{\theta}(z, \gamma) \= \frac{\theta(z)}{\theta(\gamma z)}.
	\end{equation}
From the definition it immediately follows that $j_{\theta}(z,\gamma_1\gamma_2)=j_{\theta}(z,\gamma_2)j_{\theta}(\gamma_2z,\gamma_1)$, so $j_{\theta}$ is indeed an automorphy factor on $\Gamma_{\theta}$.
We have $j_{\theta}(z, T^2) = 1$ and $j_{\theta}(z, S) = (-iz)^{-1/2}$, and in general we have $j_{\theta}(z, (\smat abcd)) = \zeta\cdot(cz+d)^{-1/2}$ for some suitable $8$-th root of unity $\zeta$ (an explicit expression for~$\zeta$ can be found in~\cite{Mu}*{Th.~7.1}).
Using this automorphy factor we define the following slash operator
in weight $k/2$ (that acts on holomorphic functions defined on the upper
half-plane~$\HH$)
	\begin{equation} \label{eq:slash1}
	\big(f|_{k/2}A\big)(z) \= j_{\theta}(z,A)^kf\Big(\frac{az+b}{cz+d}\Big),
	\end{equation}
where $A=(\smat abcd)\in\Gamma_{\theta}$. More generally, for $\eps\in\{-,+\}$ define a
slash operator $|_{k/2}^{\eps}$ by
	\begin{equation} \label{eq:slash2}
	f|_{k/2}^{\eps}A \= \chi_{\eps}(A)f|_{k/2}A,
	\end{equation}
where $\chi_{\eps}\colon\Gamma_{\theta}\to\{\pm1\}$ is the homomorphism defined
by $\chi_{\eps}(S)=\eps$ and $\chi_{\eps}(T^2)=1$. The slash operator defines
a group action, that is, $f|AB = (f|A)|B$. Another fact that we will use is
that for all $(\smat abcd)\in\SL_{2}(\RR)$ we have
	\begin{equation} \label{eq:iminv}
	\Im\Big(\frac{a\tau+b}{c\tau+d}\Big) \= \frac{\Im(\tau)}{|c\tau+d|^2}.
	\end{equation}

For any real number~$a$ we will denote by~$q^{a}$ the analytic function
	\[q^{a} \= q^{a}(z) \= \exp(2\pi i az).\]
Any $N$-periodic holomorphic function on $\HH$ admits an expansion in powers of $q^{1/N}$
(in general as a Laurent series, but in our case such expansions will have only finitely
many negative powers). We will be using subscripts to indicate the main
variable of~$q$, i.e., $q_{\tau}^a$ is the same as~$q^a(\tau)$; by default
the variable of $q^a$ is~$z$.

\subsection{Modular forms for the group $\Gamma_{\theta}$}
We begin by defining
the classical Jacobi theta series (the so-called Thetanullwerte):
\begin{align*}
	\thV(z) &\= \sum_{n\in\ZZ+\frac12}q^{\frac12n^2} \= 2\frac{\eta(2z)^2}{\eta(z)},\\
	\thU(z) &\= \sum_{n\in\ZZ}q^{\frac12n^2} \= \frac{\eta(z)^5}{\eta(z/2)^2\eta(2z)^2} \quad (=\theta(z)),\\
	\thW(z) &\= \sum_{n\in\ZZ}(-1)^nq^{\frac12n^2} \= \frac{\eta(z/2)^2}{\eta(z)},
\end{align*}
where $\eta(z) = q^{1/24}\prod_{n\ge 1}(1-q^n)$ is the Dedekind eta function.
The functions $\thV^4$, $\thU^4$, and $\thW^4$ generate the ring of
holomorphic modular forms on $\Gamma(2)$ and satisfy the Jacobi identity
\begin{equation} \label{eq:threeterm}
	\thU^4 \= \thV^4 + \thW^4.
\end{equation}
The $q$-expansions of these forms at the cusp $i\infty$ are as follows:
\begin{align*}
	\thV^4(z) &\= 16q^{1/2}+64q^{3/2}+96q^{5/2}+O(q^3),\\
	\thU^4(z) &\= 1+8q^{1/2}+24q+32q^{3/2}+24q^2+48q^{5/2}+O(q^3),\\
	\thW^4(z) &\= 1-8q^{1/2}+24q-32q^{3/2}+24q^2-48q^{5/2}+O(q^3).
\end{align*}
Under the action of $\SL_2(\ZZ)$ the theta functions transform as follows.
Under the action of~$S$ we have
\begin{equation}\label{eq:transformS}
	\begin{split}
		(-iz)^{-1/2}\thV(-1/z) &\= \thW(z),\\
		(-iz)^{-1/2}\thU(-1/z) &\= \thU(z),\\
		(-iz)^{-1/2}\thW(-1/z) &\= \thV(z),
	\end{split}
\end{equation}
and under the action of $T$ we have
\begin{equation}\label{eq:transformT}
	\begin{split}
		\thV(z+1) &\= e^{i\pi/4}\thV(z),\\
		\thU(z+1) &\= \thW(z),\\
		\thW(z+1) &\= \thU(z)
	\end{split}
\end{equation}
Together with the $q$-series for $\thV$, $\thU$, and $\thW$,
these transformations allow us to compute the~$q$-series expansion
of any expression in theta functions at any of the three cusps
of~$\Gamma(2)$.

Using these theta functions we can define the classical modular lambda invariant
    \[\lambda(z) \= \frac{\thV^4(z)}{\thU^4(z)}\= 16q^{1/2}-128q+704q^{3/2}+\dots\ ,\]
which is a Hauptmodul for $\Gamma(2)$. In particular, we have
    \[\lambda\Big(\frac{az+b}{cz+d}\Big) \= \lambda(z),\quad
    \pmat abcd \eqs \pmat 1001 \Mod{2},\]
and any meromorphic function with these transformation properties and
with appropriate behavior at the cusps can
be expressed as a rational function of $\lambda$.
From~\eqref{eq:threeterm} -- \eqref{eq:transformT} we see that under the action
of $\PSL_2(\ZZ)$ the function $\lambda(z)$ transforms as follows:
\begin{equation}\label{eq:transformlambda}
	\begin{split}
		\lambda\Big(-\frac{1}{z}\Big) &\= 1-\lambda(z),\\
		\lambda(z+1) &\= \frac{\lambda(z)}{\lambda(z)-1}.
	\end{split}
\end{equation}
Since $\thU$, $\thV$, and $\thW$ do not vanish in $\HH$ (by the product
expression in terms of $\eta(z)$), we get the well-known fact
that $\lambda(z)$ omits the values $0$ and $1$.

Using $\lambda(z)$, define a Hauptmodul $\jth$ for the group $\Gamma_{\theta}$
\begin{equation}\label{eq:jtheta}
	\jth(z) \= \frac{1}{16}\lambda(z)(1-\lambda(z)) \= \frac{\thV^4(z)\thW^4(z)}{16\thU^8(z)} \=
	q^{1/2}-24q+300q^{3/2}+\dots\ .
\end{equation}
Note that $\jth(z)=\eta(z/2)^{24}\eta(2z)^{24}\eta(z)^{-48}$,
hence it does not have zeros in $\HH$.
This function satisfies the transformation laws
\begin{equation*}\label{eq:transformjtheta}
	\begin{split}
		\jth\Big(-\frac{1}{z}\Big) &\= \jth(z),\\
		\jth(z+2) &\= \jth(z),
	\end{split}
\end{equation*}
and it maps the fundamental domain $\mathcal{D}$
conformally onto the cut plane $\CC\sm [1/64,+\infty)$.
Finally, note that $1/\jth$ vanishes at the cusp $1$ since a simple
calculation shows that
	\begin{equation}\label{eq:jthetacusp}
	\frac{1}{\jth(1-1/z)} \= -4096q-98304q^2+O(q^3).
	\end{equation}
\subsection{Asymptotic notation.}
We freely use the standard big $O$ notation. In addition, we also use
Vinogradov's ``$\ll$'' sign
	\[f \ll_{\eps,\delta,\dots} g \quad\Leftrightarrow\quad f = O_{\eps,\delta,\dots}(g).\]
Notationally, we prefer to use ``$O$'' for sequences and additive
remainders, while for most inequalities with implied constants we
use~``$\ll$''.

\section{Weakly holomorphic modular forms on $\Gamma_{\theta}$ of weight $3/2$}
\label{sec:mf}
We begin by constructing a basis
for a certain space of weakly holomorphic modular forms of weight $3/2$.
Namely, let $\{g_n^{+}(z)\}_{n\ge0}$ and $\{g_n^{-}(z)\}_{n\ge1}$ be
two collections of holomorphic functions on the upper half-plane $\HH$
that satisfy the transformation properties
	\begin{align} \label{cond:trans}
	\begin{split}
	g_n^{\eps}(z+2) &\= g_n^{\eps}(z),\\
	(-iz)^{-3/2}g_n^{\eps}(-1/z) &\= \eps g_n^{\eps}(z),
	\end{split}
	\end{align}
as well as the following behavior at the cusps
	\begin{align} \label{cond:asympt}
	\begin{split}
	&g_n^{+}(z) \= q^{-n/2}+O(q^{1/2}),\; z\to i\infty,\\
	&g_n^{-}(z) \= q^{-n/2}+O(1),\; z\to i\infty,\\
	&g_n^{\eps}(1+i/t) \to 0,\; t\to\infty.
	\end{split}
	\end{align}
The reason behind these conditions will be made clear in the next section. We
make the following ansatz:
	\begin{align} \label{mod:ansatz}
	\begin{split}
	g_n^{+}(z)& \= \theta^3(z)P_n^{+}(\jth^{-1}(z)),\\
	g_n^{-}(z)& \= \theta^3(z)(1-2\lambda(z))P_n^{-}(\jth^{-1}(z)),
	\end{split}
	\end{align}
where $P_n^{\pm}\in\QQ[x]$ are monic polynomials of degree $n$
and $P_n^{-}(0)=0$.
The polynomials $P_n^{\pm}$ are uniquely determined by the first two
conditions in~\eqref{cond:asympt}, since $\jth^{-1}$ has
$q$-expansion starting with~$q^{-1/2}+24+O(q^{1/2})$, and thus the coefficients
of $P_n^{\pm}$ can be found by inverting an upper-triangular matrix.
The transformation properties~\eqref{cond:trans}
follow from the properties of~$\jth(z)$ and~$\lambda(z)$.
The first few of these functions are
	\begin{align*}
	\begin{split}
	g_0^{+}& \= \theta^3,\\
	g_1^{+}& \= \theta^3\cdot(\jth^{-1}-30),\\
	g_2^{+}& \= \theta^3\cdot(\jth^{-2}-54\jth^{-1}+192),
	\end{split}
	\quad
	\begin{split}
	g_1^{-}& \= \theta^3\cdot(1-2\lambda)\cdot(\jth^{-1}),\\
	g_2^{-}& \= \theta^3\cdot(1-2\lambda)\cdot(\jth^{-2}-22\jth^{-1}),\\
	g_3^{-}& \= \theta^3\cdot(1-2\lambda)\cdot(\jth^{-3}-46\jth^{-2}+252\jth^{-1}).
	\end{split}
	\end{align*}
Note that the polynomials $P_{n}^{+}$ are the Faber polynomials associated to the function~$1/J$, viewed as a function on the unit disk (see~\cite{Cu}).
In the next theorem we give closed form expressions for
generating functions of $\{g_{n}^{\pm}\}$.
\begin{theorem}
	The generating functions for $\{g_n^{+}(z)\}_{n\ge0}$
	and $\{g_n^{-}(z)\}_{n\ge1}$ are given by
	\begin{align} \label{eq:kernels}
	\begin{split}
	\sum_{n=0}^{\infty} g_n^{+}(z)e^{i\pi n\tau}
	\= \frac{\theta(\tau)(1-2\lambda(\tau))\theta^3(z)\jth(z)}{\jth(z)-\jth(\tau)}
	 \es{=:} K_{+}(\tau, z),\\
	\sum_{n=1}^{\infty} g_n^{-}(z)e^{i\pi n\tau}
	\= \frac{\theta(\tau)\jth(\tau)\theta^3(z)(1-2\lambda(z))}{\jth(z)-\jth(\tau)}
	\es{=:} K_{-}(\tau, z).
	\end{split}
	\end{align}
    Here $K_{\pm}(\tau,z)$ is a meromorphic function with poles at
    $\tau\in\Gamma_{\theta}z$, and the series on the left-hand side
    converges for all large enough $\Im(\tau)$.
\end{theorem}
\begin{proof}
	The proof follows the same lines as the proof of Theorem~2 from~\cite{DJ}.
	We only prove the statement for $g_n^{+}$, since
	the case of~$g_n^{-}$ is almost identical. From the $q$-expansion of~$\jth^{-1}$ and the fact
	that
	\[
	\frac{\jth(z)}{\jth(z)-\jth(\tau)} \= \sum_{n\ge 0} \jth^n(\tau)\jth^{-n}(z),
	\]
	it is clear that the $g_n^{+}$ defined by~\eqref{eq:kernels} are also
	of the form $\theta^3(z) P_n(\jth^{-1}(z))$ for some monic polynomial $P_n$ of degree $n$. The only thing that we need to check is that they satisfy
	\[
	g_n^{+}(z) \= q^{-n/2}+O(q^{1/2}),\; z\to i\infty,
	\]
	or, equivalently, that $P_n=P_n^{+}$.
	By Cauchy's theorem we know that
	\[
	g_n^{+}(z) \= \frac{1}{2}\int_{\tau_0}^{\tau_0+2}
	K_{+}(\tau,z)q_{\tau}^{-n/2}d\tau
	\= \frac{1}{2\pi i}\oint_{C}
	K_{+}(\tau,z)q_{\tau}^{-(n+1)/2}d(q_\tau^{1/2}),
	\]
	where $\tau_0\in\HH$ has sufficiently large imaginary part and $C$
	is a small enough loop around~$0$ in the $q_{\tau}^{1/2}$-plane.
	Using the identity
	\begin{equation}\label{eq:diff}
	q_{\tau}^{1/2}\frac{d\jth}{d(q_{\tau}^{1/2})}(\tau) \=
	\frac{\jth'(\tau)}{\pi i} \= \theta^4(\tau)(1-2\lambda(\tau))\jth(\tau)
	\end{equation}
	we get that
	\[
	K_{+}(\tau,z) \= \frac{q_{\tau}^{1/2}\frac{d\jth}{d(q_{\tau}^{1/2})}(\tau)}{\jth(z)-\jth(\tau)}\cdot\frac{\theta^3(z)\jth(z)}{\theta^3(\tau)\jth(\tau)},
	\]
	and thus changing the variable of integration we get
	\[
	g_n^{+}(z) \= \frac{1}{2\pi i}\oint_{\tilde C}
	\frac{(q_{\tau}^{1/2}(j))^{-n}}{\jth(z)-j}\cdot\frac{\theta^3(z)\jth(z)}{\theta^3(\tau)j}dj.
	\]
	(We write $q_{\tau}^{1/2}(j)$ to emphasized dependence on $j$.)Now recall that $\theta^3(z)P_n^{+}(\jth^{-1}(z)) = q^{-n/2}+O(q^{1/2})$, so that $(\theta^3(\tau)P_n^{+}(j^{-1})-q_{\tau}^{-n/2}(j))/j$ is holomorphic in some small neighborhood of $0$ in the $j$-plane. Therefore, for some small loop $\tilde C$ around zero, we have
	\begin{multline*}
	g_n^{+}(z) \= \frac{1}{2\pi i}\oint_{\tilde C}
	\frac{(q_{\tau}^{1/2}(j))^{-n}}{\jth(z)-j}\cdot\frac{\theta^3(z)\jth(z)}{\theta^3(\tau)j}dj
	\=
	\frac{\theta^3(z)}{2\pi i}\oint_{\tilde C}
	\frac{P_n^{+}(j^{-1})}{j\big(\jth(z)-j\big)}{\jth(z)}dj
	\\
	\=
	-\frac{\theta^3(z)}{2\pi i}\oint_{\tilde C}
	\frac{P_n^{+}(j^{-1})}{\big(j^{-1}-\jth^{-1}(z)\big)}dj^{-1}
	\= \theta^3(z)P_n^{+}(\jth^{-1}(z)).
	\end{multline*}
	The last sign is changed since the contour
	for $j^{-1}$ in the last application of Cauchy's formula has the
	opposite orientation.
\end{proof}
\textbf{Remark.} From~\eqref{eq:diff} it also follows that
$K_{\eps}(\tau,z)$ has a simple pole at~$z=\tau$ with residue~$\frac1{i\pi}$
for all $\tau\in\HH$. We also record here the following identities
for~$K_{\eps}$:
	\begin{equation}
	\label{eq:kerneltrans}
	\begin{split}
	K_{\eps}(\tau,-1/z)& \= \eps(-iz)^{3/2} K_{\eps}(\tau,z),\\
	K_{\eps}(-1/\tau,z)& \= -\eps(-i\tau)^{1/2} K_{\eps}(\tau,z).
	\end{split}
	\end{equation}
Note that generating functions very similar to~\eqref{eq:kernels}
have also been used in~\cite{Za1} in the computation of traces of singular
moduli.

\section{Interpolation basis for even functions}
\label{sec:basis}
Let us define a function $b_m^{\eps}\colon\RR\to\RR$ by the integral
	\begin{equation}
	\label{eq:bm_definition}
	b_m^{\eps}(x) \= \frac12\int_{-1}^{1} g_m^{\eps}(z)e^{i\pi x^2 z}\dz,
	\end{equation}
where the path of integration is chosen to lie in the upper
half-plane and orthogonal to the real line at the endpoints~$1$
and~$-1$. Since $g_m^{\eps}$ has exponential decay at $\pm 1$, the above integral
converges.
Note that $b_m^{\eps}$ is defined for $m\ge 0$ if $\eps=+1$ and for $m\ge 1$ if $\eps=-1$;
for convenience let us also define $b_{0}^{-}(x) = 0$.

Recall that Schwartz functions are $C^{\infty}$-smooth functions that,
together with all of their derivatives, decay faster than any inverse
power of $x$.

\begin{proposition} \label{prop:schwartz_eigenfunction}
The function $b_m^{\eps}\colon\RR\to\RR$ is an even Schwartz function that satisfies
	\[\ftt{b_m^{\eps}}(x) \= \eps b_m^{\eps}(x)\]
and
	\[b_m^{\eps}(\sqrt{n}) \= \delta_{n,m},\quad n\ge 1,\,m\ge0,\]
where $\delta_{n,m}$ is the Kronecker delta. In addition, we have $b_0^{+}(0) = 1$.
\end{proposition}
\begin{proof}
Clearly, $b_m^{\eps}$ is an even function, since $e_z(x) = e^{i\pi x^2 z}$ is even.
That it indeed takes real values for $x\in\RR$ can be seen by taking the integral over the semicircle $z=e^{it}$, $t\in(0,\pi)$, making a change of variables $z\mapsto -\overline{z}$, and noting that $\overline{g_m^{\eps}(z)} = g_m^{\eps}(-\overline{z})$.
Let us prove that $b_m^{\eps}$ belongs to the Schwartz class.
We will only consider the case ``$\eps=+$'', but the same argument will work
also in the case ``$\eps=-$''. Since $g_n^{+}(z) = \theta^3(z)P^+_n(\jth^{-1}(z))$,
it is enough to prove that for each $n\in \NN$ the integral
	\[\beta_n(x) \= \frac12\int_{-1}^{1} \theta^3(z)\jth^{-n}(z)e^{i\pi x^2 z}\dz\]
is a Schwartz function. On the circle arc from $-1$ to $1$ the function $1/\jth(z)$ takes real
values between $0$ and $64$, and moreover
	\[\jth^{-1}(\pm 1+i/t) \,\le\, C\exp(-2\pi t),\; t\to \infty,\; \Re(t) > 0.\]
By taking the $k$-th derivative of $\beta_n(x)$ with respect to $x$ under the
integral we obtain
	\[\beta_n^{(k)}(x) \= \frac12\int_{-1}^{1} \theta^3(z) \jth^{-n}(z)Q_{k}(x,z)e^{i\pi x^2 z}\dz,\]
where $Q_{k}(x,z)$ are polynomials defined by
    \begin{equation}\label{eq:polyqk}
    \frac{\partial^k}{\partial x^k}e^{i\pi x^2z} = Q_{k}(x,z)e^{i\pi x^2z}.
    \end{equation}
Clearly, there exists a constant $C_k$ such that
	\[|Q_{k}(x,z)| \leqs C_k(1+|x|^2)^{k}(1+|z|^2)^{k},\]
thus we get
	\[|\beta_n^{(k)}(x)| \leqs \pi 2^{k+3}C_k(1+|x|^2)^{k}\int_{0}^{1/2}
	\jth^{-n}(e^{i\pi t})e^{-\pi x^2 \sin(\pi t)}\dt.\]
Here we used a rather crude estimate $|\theta(e^{i\pi t})| < 2$ for $t\in(0,1/2)$. When $|x|$ is
small, we estimate the above integral by $64^n$, for all other values of $x$ we estimate the integral
by splitting it into two parts (where we take $\delta = (\sqrt{\pi}x)^{-1}$):
\begin{align*}
\int_{0}^{1/2} \jth^{-n}(e^{i\pi t})e^{-\pi x^2 \sin(\pi t)}\dt &\=
\int_{0}^{\delta} \jth^{-n}(e^{i\pi t})e^{-\pi x^2 \sin(\pi t)}\dt +
\int_{\delta}^{1/2} \jth^{-n}(e^{i\pi t})e^{-\pi x^2 \sin(\pi t)}\dt \\
	&\leqs C\delta e^{-2/\delta} + 64^n e^{-2\pi\delta x^2}
	\= e^{-2\sqrt{\pi}x}(64^n+C/(x\sqrt{\pi})),
\end{align*}
from which it follows that $\beta_n$ is a Schwartz function.

To check that $\ftt{b_m^{\eps}} = \eps b_m^{\eps}$ we will use the fact
that $\ftt{e_z} = (-iz)^{-1/2}\,e_{-1/z}$ and the transformation
property~\eqref{cond:trans}:
	\begin{align*}
	\widehat{b_m^{\eps}}(x) &\= \frac12\int_{-1}^{1} g_m^{\eps}(z)(-iz)^{-1/2}e^{i\pi x^2(-1/z)}\dz \\
	&\= \frac12\int_{-1}^{1} -g_m^{\eps}(z)(-iz)^{3/2}e^{i\pi x^2(-1/z)}d(-1/z)\\
	&\= \frac12\int_{1}^{-1} \eps g_m^{\eps}(-1/z)e^{i\pi x^2(-1/z)}d(-1/z)
	\= \eps b_m^{\eps}(x).
	\end{align*}
	In the above computations we always choose the branch of $(-iz)^{k/2}$ that takes positive values for $z$ on the imaginary semiaxis. Finally, note that
		\[b_m^{\eps}(\sqrt{n}) \= \frac12\int_{-1}^{1} g_m^{\eps}(z)e^{i\pi n z}\dz\]
	is simply the coefficient of $q^{-n/2}$ in the $q$-expansion of $g_m^{\eps}$, so that~\eqref{cond:asympt} immediately implies $b_m^{\eps}(\sqrt{n}) = \delta_{n,m}$
	and $b_0^{+}(0) = 1$.
\end{proof}
\noindent\textbf{Remark.} Note that~\eqref{cond:asympt} also implies that
$b_m^+(0) = \delta_{m,0}$, and using the explicit
formula~\eqref{eq:kernels} for the kernel $K_{-}$, we also get
	\begin{equation*}
	b_m^-(0) \= \begin{cases}
	-2,\quad m\ge1\mbox{ is a square},\\
	0,\quad\quad \mbox{ otherwise}.	
	\end{cases}
	\end{equation*}
Alternatively, this last equation follows from the Poisson summation formula
	\[\sum_{n\in\ZZ} b_{m}^{-}(n) \= \sum_{n\in\ZZ} \ftt{b_{m}^{-}}(n)
	 \= -\sum_{n\in\ZZ} b_{m}^{-}(n).\]

To establish other properties of the sequences $\{b_m^{\eps}(x)\}_m$
we will need to work with generating functions. Let $\mathcal{D}$ be the
standard fundamental domain for the group $\Gamma_{\theta}$ (as defined
in~\eqref{eq:funddom}). For a fixed $x$ define a function
$F_{\eps}(\tau, x)$ on the set
	\[\{\tau\in\HH\: \forall k\in\ZZ,\ |\tau-2k| > 1 \} \es{\supset}\mathcal{D}+2\ZZ\]
by
	\begin{equation} \label{eq:genfun}
	F_{\eps}(\tau, x) \= \frac12\int_{-1}^{1}K_{\eps}(\tau,z)e^{i\pi x^2 z}\dz,
	\end{equation}
where the contour is the semicircle in the upper half-plane that passes
through~$-1$ and~$1$. Note that for $\Im(\tau) > 1$ we have
	\begin{equation} \label{eq:genfunseries}
	F_{\eps}(\tau, x) \= \sum_{n=0}^{\infty} b_n^{\eps}(x)e^{i\pi n \tau},
	\end{equation}
and the series converges absolutely. Our next task is to show that
$F_{\eps}$ can be analytically continued to $\HH$
(and hence~\eqref{eq:genfunseries} also holds for all $\tau\in\HH$).
\begin{proposition} \label{prop:genfunceq}
	For any $\eps\in\{+,-\}$ and $x\in\RR$ the function $F_{\eps}(\tau,x)$
	admits an analytic continuation to $\HH$. Moreover, the analytic continuation
	satisfies the functional equations
	\begin{equation}\label{eq:genfun_main}
		\begin{split}
			F_{\eps}(\tau,x) - F_{\eps}(\tau+2,x) &\= 0,\\
			F_{\eps}(\tau,x) + \eps(-i\tau)^{-1/2}F_{\eps}\Big(-\frac{1}{\tau},x\Big) &\=
			e^{i\pi \tau x^2} + \eps(-i\tau)^{-1/2}e^{i\pi (-1/\tau) x^2}.
		\end{split}
	\end{equation}
\end{proposition}
\begin{proof}
	To prove the theorem, it is enough to show that there exists an
    analytic continuation to some open set $\Omega$ containing the
    relative closure of $\mathcal{D}$, on which~\eqref{eq:genfun_main}
    holds. Indeed, we can then choose a smaller $\Omega$ in such a
    way that $\Omega \cap \gamma^{-1}(\Omega)\ne\varnothing$ if and
    only if $\gamma\in\{T^2,T^{-2},S,I\}$.
    Since $\cup_{g\in\Gamma_{\theta}}g\Omega = \HH$, we can
    construct a continuation inductively by repeatedly
    using~\eqref{eq:genfun_main} to pass to the
    neighboring sets $g\Omega$ that have not been covered yet.
    Since $\Gamma_{\theta}$ is generated by $S$ and $T^2$, and
    the only relation is $S^2=1$, this
    process indeed gives a single-valued analytic continuation
    (the main reason is that there are no cycles of neighboring
    domains; this is also clear from
    Figure~\ref{figure:fundamental_domain}).

	The first functional equation in~\eqref{eq:genfun_main} is clearly satisfied, since the
	integral that defines $F_{\eps}$ automatically defines a $2$-periodic function on
	the open set $\{\tau\in\HH \: \forall k\in\ZZ,\ |\tau-2k| > 1 \}$ that
	contains the vertical lines $\Im(\tau)=\pm1$.
	
	Hence, we only need to deal with the second functional equation.
	We can get an analytic continuation of $F_{\eps}$ to some neighborhood of
	$\{z\in\HH\:|z|=1, z\ne i\}$ by changing the contour of integration
	in~\eqref{eq:genfun}. First, we rewrite the integral as
	\begin{equation} \label{eq:genfun2}
		\begin{split}
			2F_{\eps}(\tau, x) &\= \int_{-1}^{i}K_{\eps}(\tau,z)e^{i\pi x^2 z}\dz
			+ \int_{i}^{1}K_{\eps}(\tau,z)e^{i\pi x^2 z}\dz \\
			&\= \int_{-1}^{i}K_{\eps}(\tau,z)e^{i\pi x^2 z}\dz
			- \int_{-1}^{i}K_{\eps}(\tau,-1/z)e^{i\pi x^2 (-1/z)}z^{-2}\dz \\
			&\= \int_{-1}^{i}K_{\eps}(\tau,z)
			(e^{i\pi x^2 z}+\eps (-iz)^{-1/2}\,e^{i\pi x^2 (-1/z)})\dz,
		\end{split}
	\end{equation}
	where we have used the transformation property~\eqref{eq:kerneltrans}.
	Note, that if $\tau$ belongs to $\ol{\mathcal{D}}\cup \ol{S\mathcal{D}}$,
	then the only poles of $K_{\eps}(\tau,z)$ (as a function of $z$)
	inside $\ol{\mathcal{D}}\cup \ol{S\mathcal{D}}$ are at $z=\tau$ and $z=-1/\tau$.
	Let $\gamma_1$ be the circle arc from $-1$ to $i$, and let $\gamma_2$ be a simple
	smooth path from $-1$ to $i$ that lies
	inside $\ol{S\mathcal{D}}$ and strictly below $\gamma_1$. Denote by
	$\mathcal{F}$ the region enclosed between $\gamma_1$ and $\gamma_2$. We will now
	build a continuation of $F_{\eps}$ to $\mathcal{F}$ and show that it satisfies
	the functional equation. We define a continuation by the contour integral
	\[\tilde F_{\eps}(\tau,x) \= \frac12\int_{\gamma_2}K_{\eps}(\tau,z)
	(e^{i\pi x^2 z}+\eps (-iz)^{-1/2}e^{i\pi x^2 (-1/z)})\,\dz.\]
	\begin{figure}
		\centering
		\begin{tikzpicture}
		\definecolor{cv0}{rgb}{0.9,0.9,0.9}
		\draw (-3,0)node[below] {$-1$};
		\draw (3,0) node[below] {$1$};
		\draw (0,3) node[above,xshift=5] {$i$};
		\draw[lightgray] (3,0) arc (180:120:3);
		\draw[lightgray] (-3,0) arc (0:60:3);
		\draw[lightgray] (3,0) -- (3,6);
		\draw[lightgray] (-3,0) -- (-3,6);
		\draw[lightgray] (-4.5,0) -- (4.5,0);
		\draw[lightgray] (3,0) arc (0:180:3);
		\filldraw[dashed, color=cv0, line width=0.35mm, draw=black] (-3,0) arc (180:90:3) to[out=200,in=70] node[above, midway, xshift=-5,yshift=20, black]{$\gamma_1$} node[below, midway,yshift=-5,black]{$\gamma_2$} (-3,0);
		\draw (-2.05,1.6) node[above]{$\mathcal{F}$};
		\draw (0,5) node[above]{$\mathcal{D}$};
		\fill (-1.4,2.55) circle (0.04) node[above] {$\tau$};
		\fill (1.489,2.712) circle (0.04) node[above] {$S\tau$};
		\end{tikzpicture}
		\caption{Fundamental domain for $\Gamma_{\theta}$
			and the contour of integration.}
		\label{figure:domains}
	\end{figure}
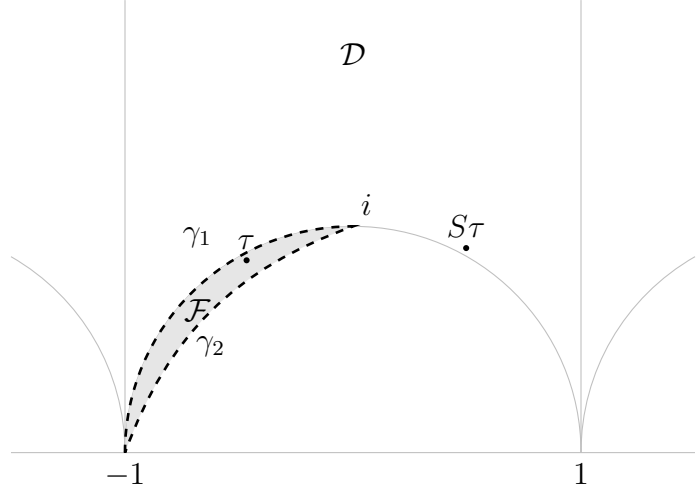

	Clearly, $F_{\eps} = \tilde F_{\eps}$ for $\tau$ with big enough imaginary part, so
	$\tilde F$ indeed defines an analytic continuation to $\mathcal{F}$.
	For $\tau \in \mathcal{F}$ we compute
	\begin{align*}
		\tilde F_{\eps}(\tau,x) + \frac{\eps}{\sqrt{-i\tau}}F_{\eps}\Big(-\frac{1}{\tau},x\Big) &=
		\tilde F_{\eps}(\tau,x) - \frac12\int_{\gamma_1}\frac{\eps K_{\eps}(-1/\tau,z)}{-\sqrt{-i\tau}}
		(e^{i\pi x^2 z}+\frac{\eps e^{i\pi x^2 (-1/z)}}{\sqrt{-iz}})\dz \\
		&=
		\tilde F_{\eps}(\tau,x) - \frac12\int_{\gamma_1}K_{\eps}(\tau,z)
		(e^{i\pi x^2 z}+\eps (-iz)^{-1/2}e^{i\pi x^2 (-1/z)})\dz\\
		&=
		\frac12\int_{\partial\mathcal{F}}K_{\eps}(\tau,z)
		(e^{i\pi x^2 z}+\eps (-iz)^{-1/2}e^{i\pi x^2 (-1/z)})\dz \\
		&= i\pi \sum_{w\in\mathcal{F}} {\rm Res}_{z=w}\big(K_{\eps}(\tau,z)
		(e^{i\pi x^2 z}+\eps (-iz)^{-1/2}e^{i\pi x^2 (-1/z)})\big) \\
		&= e^{i\pi x^2 \tau}+\eps (-i\tau)^{-1/2}e^{i\pi x^2 (-1/\tau)},
	\end{align*}
	which is precisely the functional equation that we needed. Similar
    computation works for the arc from $i$ to $1$. The only thing that
    is left is to check that $F_{\eps}$ has no singularity
    at~$\tau = i$. For~$\eps = 1$ this follows from the second
    functional equation, while for $\eps = -1$ both $2\lambda(z)-1$
	and $e^{i\pi z r^2} +\eps (-iz)^{-1/2}e^{i\pi (-1/z) r^2}$ vanish
    at $z=i$, so that they cancel the double pole at $i$ coming
    from $\jth(i)-\jth(z)$, and hence the integral~\eqref{eq:genfun2}
    converges at~$\tau=i$.
\end{proof}
As an immediate corollary, we obtain that
formula~\eqref{eq:genfunseries} is valid
for all $\tau\in\HH$. This already implies that for all $\delta>0$
we have $b_n^{\eps}(x) = O((1+\delta)^n)$. In the next section we prove
a much stronger estimate.

Note that the only properties of $K_{\eps}$ that were used in the proof
are the modularity in~$\tau$ and in~$z$, as well as the fact that the only poles are at $z \in \Gamma_{\theta}\tau$, and that the residue
at~$z=\tau$ is equal to~$1/(i\pi)$.

\section{Growth estimate} \label{sec:mainest}
The main result of this section is the following.
\begin{theorem} \label{thm:mainest}
	For any $\eps\in\{+,-\}$ the numbers $b_n^{\eps}(x)$ satisfy
	\[|b_n^{\eps}(x)| \= O(n^{2})\]
	uniformly in $x$.
\end{theorem}
To prove this we will use the following general result that goes back to Hecke
(see, for example,~\cite{BeKn}*{Lemma~2.2, (ii)}).
\begin{lemma} \label{lem:heckebound}
	If a $2$-periodic analytic function $f\colon\HH\to\CC$ has a Fourier expansion $f(\tau) = \sum_{n\ge 0}a_n e^{i\pi n\tau}$ and for some $\alpha>0$ it satisfies
	\[ |f(\tau)| \leqs C\,\Im(\tau)^{-\alpha}\;\mbox{ for }\ \Im(\tau)<c,\]
	then for all sufficiently large $n$ we have
	\[|a_n| \leqs C\,\big(\frac{e\pi}{\alpha}\big)^{\alpha}\, n^\alpha.\]
\end{lemma}
To prove Theorem~\ref{thm:mainest} we will apply this lemma to the
generating function $F_{\eps}(\tau,x)$. To simplify notation, we will
write $F_{\eps}(\tau)$ instead of $F_{\eps}(\tau,x)$.
The estimate of $|F_{\eps}(\tau)|$ naturally splits into two parts:
combinatorial (estimating $F_{\eps}(\tau) - (F_{\eps}|A)(\tau)$
using functional equations) and  analytic (estimating
$F_{\eps}(\tau)$ using the defining contour integral).

To deal with the first part, we define functions $\phi_{A}(\tau)$
for $A\in\Gamma_{\theta}$:
    \begin{equation} \label{eq:funceq}
    \phi_{A}(\tau) \es{:=} F_{\eps}(\tau) - (F_{\eps}|_{1/2}^{-\eps}A)(\tau).
    \end{equation}
From the functional equations~\eqref{eq:genfun_main} for $F_{\eps}$
we have
	\begin{equation}
	\label{eq:cocycledef}
	\begin{split}
	\phi_{T^2}(\tau) &\= 0,\\
	\phi_{S}(\tau) &\=
	e^{i\pi x^2\tau} + \eps(-i\tau)^{-1/2}e^{i\pi x^2(-1/\tau)}\,.
	\end{split}
	\end{equation}
Moreover, the functions $\phi_A$ satisfy the cocycle relation
$\phi_{AB}=\phi_{B}+\phi_{A}|B$ (where we write~$|$
for~$|_{1/2}^{-\eps}$).
In other words, the collection $\{\phi_A\}_{A\in\Gamma_{\theta}}$
forms what is usually called a~$\Gamma_{\theta}$-cocycle
(see, for example,~\cite{Kn2}).

First, we need the following elementary lemma.
\begin{lemma} \label{lem:contfrac}
	For any $\tau\in\HH$ with $|\tau|\ge 1$ and any sequence of non-zero integers $\{n_j\}_{j\ge 1}$
	define a sequence of numbers $\tau_{j} \in \HH$ as follows:
	\begin{align*}
		\tau_0 &\= \tau, \\
		\tau_j &\= 2n_j - \frac{1}{\tau_{j-1}},\quad j\ge 1.
	\end{align*}
	Then the sequence $\{\Im(\tau_j)\}_{j\ge 0}$ is strictly decreasing and $\Im(\tau_j) \le \frac{1}{2j-1}$ for all $j\ge 1$.
\end{lemma}
\begin{proof}
	First, observe that $|\tau_j| > 1$ for all $j\ge 1$ (the proof is by
	induction). The inequality $\Im(\tau_j)\ge \Im(\tau_{j+1})$ the follows
	from $\Im(\tau_{j+1}) = \Im(\tau_j)/|\tau_j|^2 < \Im(\tau_j)$.
	
	For $a,b\in\RR$ denote by $D(a,b)$ the half-disk with center $(a+b)/2$ whose
	boundary semicircle passes through $a$ and $b$. Let $D$ be any such
	half-disk that does not intersect
	$D(-1,1)$ and set $D' = SD$. Then a simple calculation shows that
		\[{\rm diam}(D') \leqs \frac{{\rm diam}(D)}{1+{\rm diam}(D)}.\]
	Note that $\tau_1 \in \bigcup_{n \ne 0} D(2n-1,2n+1)$, so $\tau_1$ lies in
	some half-disk of diameter $2$. Denote this half-disk by $D_1$, and
	define $D_{j+1} = 2n_j+SD_j$.
	Then $\tau_j\in D_j$ and no $D_j$ intersects $D(-1,1)$.
	By repeatedly applying the above inequality we get that $D_j$
	has diameter at most $2/(2j-1)$, thus $\Im(\tau_j)\le 1/(2j-1)$.
\end{proof}

The following lemma allows us to estimate values of certain cocycles.
\begin{lemma} \label{lem:cocyclelemma}
	Let $\{\psi_{A}\}_{A\in\Gamma_{\theta}}$ be a cocycle (with respect to
	$|\;:=\;|_{k/2}^{-\eps}$) such that
		\begin{align*}
		\psi_{T^2} &\= 0,\\
		|\psi_{S}(\tau)| &\leqs |\tau|^{\alpha} + \Im(\tau)^{-\beta}
		\end{align*}
	for some $\alpha, \beta \ge 0$.
	Let $\tau'\in\mathcal{D}$, $A\in\Gamma_{\theta}$, and $\tau = A\tau'\in\HH$
	and suppose that $\Im(\tau) \le 1$. Then
	\[|\psi_{A}(\tau')| \leqs
	|\tau|^{\alpha}+\Im(\tau)^{-\alpha-1}+2\,\Im(\tau)^{-\beta-1}.\]
\end{lemma}
\begin{proof}
	Let us consider the case when
	\[A \= ST^{2n_m}ST^{2n_{m-1}}S\dots T^{2n_1}S.\]
	By applying the cocycle relation repeatedly, we get that
	\[\psi_A \= \psi_S + \psi_S|A_1 +  \psi_S|A_2
	+ \dots + \psi_S|A_m,\]
	where we write $A_j = T^{2n_j}S\dots T^{2n_1}S$. Hence
	\[|\psi_A(\tau')| \leqs \sum_{j=0}^{m}\frac{|\psi_S(\tau_j)|}{|c_j\tau'+d_j|^{k}},\]
	where $A_j = (\smat {a_j}{b_j}{c_j}{d_j})$ and $\tau_j$ are defined by
	\begin{equation*}
		\begin{cases}
			\tau_0 \= \tau',\\
			\tau_{j} \= 2n_j - 1/\tau_{j-1}.
		\end{cases}
	\end{equation*}
	Under these definitions $\tau_j = \frac{a_j\tau'+b_j}{c_j\tau'+d_j}$
	and $\tau=-1/\tau_m$. Multiplying both sides of the above inequality by $\Im(\tau')^{k/2}$ we get
		\[\Im(\tau')^{k/2}|\psi_A(\tau')| \leqs \sum_{j=0}^{m}
		\Im(\tau_j)^{k/2}|\psi_{S}(\tau_j)|\]
	Lemma~\ref{lem:contfrac} implies that $\Im(\tau)^{-1}\ge 2m-1$ and
	$\Im(\tau_j)\ge \Im(\tau)$ for $j=0,\dots,m$. We also have
	$|\tau_j|\le \Im(\tau)^{-1}$ for $j=0,\dots,m-1$,
	since $\Im(\tau)\le \Im(\tau_{j+1}) = \Im(\tau_{j})/|\tau_j|^2\le |\tau_j|^{-1}$. Therefore
		\begin{multline*}
		\Im(\tau')^{k/2}|\psi_A(\tau')| \leqs \sum_{j=0}^{m}
		\Im(\tau_j)^{k/2}(|\tau_j|^{\alpha}+\Im(\tau_j)^{-\beta})
		\leqs
		\Im(\tau')^{k/2}\sum_{j=0}^{m}(|\tau_j|^{\alpha}+\Im(\tau_j)^{-\beta})
		\\
		\leqs \Im(\tau')^{k/2}(|\tau|^{\alpha}+m\,\Im(\tau)^{-\alpha}+(m+1)\Im(\tau)^{-\beta})
		\\ \leqs \Im(\tau')^{k/2}(|\tau|^{\alpha}+\Im(\tau)^{-\alpha-1} +2\,\Im(\tau)^{-\beta-1}),
		\end{multline*}
	where in the last line we used $m+1\le 4m-2 \le \Im(\tau)^{-1}$.
	
	The proof in the other cases (i.e., when $A$ is of the
	form $T^{2n_k}ST^{2n_{k-1}}S\dots T^{2n_1}S$,  $ST^{2n_k}ST^{2n_{k-1}}S\dots T^{2n_1}$, or  $T^{2n_k}ST^{2n_{k-1}}S\dots T^{2n_1}$)
	can be completed using similar estimates.
\end{proof}

Next, we deal with the analytic part of the estimate.
For Theorem~\ref{thm:summation} the case $n=0$ of the lemma
below will suffice, but we need the general form for the proof of
Theorem~\ref{thm:isomorphism}.
\begin{lemma} \label{lem:kernel_estimate}
	For each $n,k\ge 0$ there exists an absolute constant $C_{n,k}>0$
	such that the inequality
	\[\big|x^{k}\frac{d^n}{dx^n}\,F_{\eps}(\tau,x)\big| \leqs C_{n,k}(1+\Im(\tau)^{-(n+k+1)/2})\]
	holds for all $\tau \in\mathcal{D}$.
\end{lemma}
\begin{proof}
Let $\tau$ be any point in~$\mathcal{D}$. Since $F_{\eps}(it)$ is real
for all $t>0$, from the Schwarz reflection principle we get that
	\begin{equation}\label{eqn: reflection}
    F_{\eps}(-\ol{\tau}) \= \ol{F_{\eps}(\tau)}.
	\end{equation}
Using this symmetry we reduce the inequality to the
case $\tau\in\mathcal{D}_1$, where
$\mathcal{D}_1 = \{\tau\in\mathcal{D} \: \Re(\tau) \in (-1,0)\}$.
Observe that $\Im(\jth(\tau)) < 0$ for all $\tau\in\mathcal{D}_1$
and~$\Im(\jth(\tau)) \ge 0$ for
all $\tau\in\mathcal{D}\sm\mathcal{D}_1$.
Indeed, since $J$ is a Hauptmodul, the map $J:\mathcal{D}\to\CC$ is
injective. The identity~\eqref{eqn: reflection} for $\jth$ implies that
for $\tau\in\mathcal{D}$ the value $\jth(\tau)$ is real if and only
if $\tau$ lies on the imaginary axis. It is easy to see
from~\eqref{eq:jtheta} that $\Im(\jth(\tau))<0$
for $\tau\in\mathcal{D}_1$ and $\Im(\tau)\gg1$.
Hence, this inequality also holds for all $\tau\in\mathcal{D}_1$.
	
Define
	\[L \= \{w\in\CC \;|\; \Re(w)=\jth(i)=1/64,\ \Im(w)>0\},\]
and let $\ell$ be the preimage of $L$ under the
map $\jth\colon\mathcal{D}\to\CC$
(see Figure~\ref{fig: deforming the countour}).
Then $\ell$ is a smooth path contained
in $\mathcal{D}\sm \mathcal{D}_1$ and goes from $i$ to $1$.
We set $\gamma$ to be the path $S\ell\cup \ell$ that goes
from $-1$ to $1$. Note that $|z|$ and $|z|^{-1}$ are bounded
on $\gamma$ and that $\gamma$ has finite length (this fact will
follow from the computations below).

As in the proof of Proposition~\ref{prop:schwartz_eigenfunction}
let~$Q_n(x,z)$ be a polynomial defined by~\eqref{eq:polyqk}.
We have
	\[x^k\frac{d^n}{dx^n}\,F_{\eps}(\tau,x)=
    \frac{1}{2}\int_{-1}^{1}
    K_{\eps}(\tau,z)\,x^k\,Q_n(x,z)\,e^{i\pi x^2 z}\dz.\]
From~\eqref{eq:kerneltrans} we find
	\[x^k\frac{d^n}{dx^n}\,F_{\eps}(\tau,x)=\frac{1}{2}\int_{i}^{1}K_{\eps}(\tau,z)\,x^k\,\big(Q_n(x,z)\,e^{i\pi x^2 z}+\eps(-iz)^{-1/2}Q_n(x,-1/z)\,e^{i\pi x^2 (-1/z)}\big)\,\dz.\]
Without loss of generality, we may assume $x\geq0$. Since $|z|$ is
bounded for $z\in\gamma$, any monomial $z^\alpha x^\beta$
with $0\le \beta\le n$ is majorized by~$1+x^n$, and thus for all
such $z$ we have~$|x^k\,Q_n(x,z)|\ll_{n,k,\gamma} 1+x^{n+k}$. Then
	\begin{align} \label{eq:kernel_estimate}
	\big|x^k\,\frac{d^n}{dx^n}\,F_{\eps}(\tau,x)\big|
    \es{\ll}
	\int_\ell |K_{\eps}(\tau,z)\,x^k|\,\big| Q_n(x,z)\,e^{i\pi x^2 z}+\eps(-iz)^{-1/2}\,Q_n(x,-1/z)\,e^{i \pi x^2(-1/z)}\big|\,|\dz|
	\notag\\
    \es{\ll}
	\int_\ell |K_{\eps}(\tau,z)|\,(1+x^{k+n})\,\big(e^{-\pi x^2 \Im(z)}+|z|^{-1/2}\,e^{-\pi x^2\Im(-1/z)}\big)\,|\dz|.
	\end{align}
Next, we observe that
	\[(1+x^{k+n})\,e^{-\pi x^2 \Im(z)}
	\es{\ll_{k+n}} 1+\Im(z)^{\frac{-k-n}{2}}.\]
Note, that $1\leq|z|\ll1$ for $z\in\ell$. Hence, we get
    \begin{align} \label{eq:kernel_estimate1}
	\big|x^k\,\frac{d^n}{dx^n}\,F_{\eps}(\tau,x)\big|&
    \es{\ll}
	\int_\ell |K_{\eps}(\tau,z)|\,
    \big(1+\Im(z)^{\frac{-k-n}{2}}+|z|^{-1/2}\,
    \Im(\tfrac{-1}{z})^{\frac{-k-n}{2}}\big)\,|\dz|
	\notag\\
    &\es{=}
    \int_\ell |K_{\eps}(\tau,z)|\,
    \big(1+\Im(z)^{\frac{-k-n}{2}}+|z|^{k+n-1/2}\,
    \Im(z)^{\frac{-k-n}{2}}\big)\,|\dz|
    \notag \\
    & \es{\ll}\int_\ell
    |K_{\eps}(\tau,z)|\,\big(1+\Im(z)^{\frac{-k-n}{2}}\big)\,|\dz|.
    \end{align}
	
Without loss of generality, we may also assume that~$|\tau-i|\ge 1/10$,
since we can recover the inequality of the Lemma in the
region~$|\tau-i|<1/10$ by applying the maximum modulus principle
together with the functional equation for $F_{\eps}$.
	
For $\tau$ with $\Im(\tau)\ge1/2$ and $|\tau-i|>1/10$ we can
estimate $|K_{\eps}(\tau,z)| \ll |\theta(z)|^3$ with a constant
independent of $\tau$. Since $|\theta^3(z)|$ behaves
like $\Im(z)^{-2}e^{-\pi/\Im(z)}$ as $z$ approaches~$1$, by
splitting the integral into~$\{z\colon \Im(z)\ge1/x\}$
and~$\{z\colon\Im(z) < 1/x\}$ we obtain
	\[|F_{\eps}(\tau,x)| \es{\ll} (1+x^2)e^{-c\pi x},\]
which clearly implies the needed inequality.
    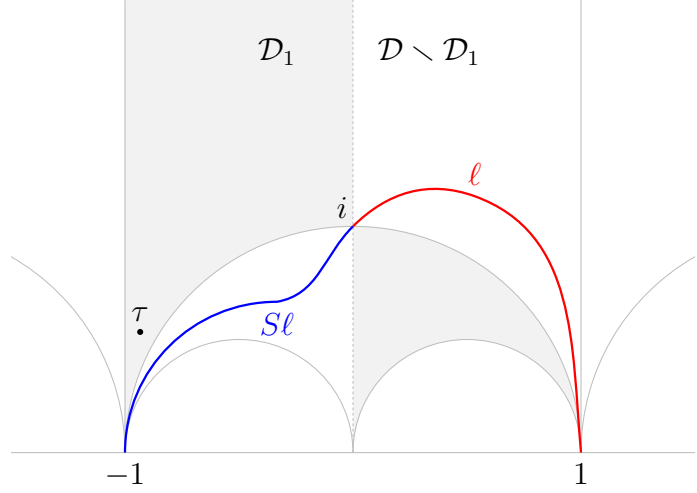
\begin{figure}
	\centering
	\begin{tikzpicture}
	\definecolor{cv0}{rgb}{0.95,0.95,0.95}
	\fill[color=cv0] (-3,0) arc (180:90:3) -- (0,6) -- (-3,6);
	\fill[color=cv0] (0,0) arc (180:0:1.5) arc (0:90:3) -- (0,0);
	\draw (-3,0)node[below] {$-1$};
	\draw (3,0) node[below] {$1$};
	\draw (0,3) node[xshift=-4,yshift=7] {$i$};
	\draw[dash pattern=on 1pt off 1pt, color=lightgray] (0,0) -- (0,6);
	\draw[lightgray] (3,0) arc (180:120:3);
	\draw[lightgray] (-3,0) arc (0:60:3);
	\draw[lightgray] (3,0) -- (3,6);
	\draw[lightgray] (-3,0) -- (-3,6);
	\draw[lightgray] (-4.5,0) -- (4.5,0);
	\draw[lightgray] (3,0) arc (0:180:3);
	\draw[lightgray] (3,0) arc (0:180:1.5);
	\draw[lightgray] (-3,0) arc (180:0:1.5);
	\draw (-1,5) node[above]{$\mathcal{D}_1$};
	\draw (1,5) node[above]{$\mathcal{D}\sm \mathcal{D}_1$};
	\fill (-2.8,1.6) circle (0.04) node[above] {$\tau$};
	\draw[color=blue, line width=0.3mm] (-3,0) to[out=90,in=180] (-1,2) node[below]{$S\ell$} to[out=10,in=225] (0,3);
	\draw[color=red, line width=0.3mm] (0,3) to[out=45,in=160] (1.6,3.4) node[above]{$\ell$} to[out=-20,in=95] (3,0);
	\end{tikzpicture}
	\caption{Deforming the contour of integration.}
    \label{fig: deforming the countour}
    \end{figure}
	
Now let $\Im(\tau)<1/2$. To bound $|K_{\eps}(\tau,z)|$ we use the
following estimates
	\begin{align*}
	|\theta(z)|& \es{\ll} |\jth(z)|^{-1/8}\,\Im(z)^{-1/2},\\
	|1-2\lambda(z)|& \es{\ll} |\jth(z)|^{1/2},
	\end{align*}
which hold for all $z\in\mathcal{D}$ near the cusp $1$ (such~$z$
correspond to large values of~$|\jth(z)|$).
The first inequality follows from the fact that $\theta^8(z)\jth(z)$
is a holomorphic modular form of weight~$4$ for~$\Gamma_{\theta}$ (the
term~$\Im(z)^{-1/2}$ comes from the modular transformation).
To prove the second inequality, simply note
that $(1-2\lambda(z))^2 \= 1-64\jth(z)$. Thus, we get
	\begin{align}
	|K_{+}(\tau, z)| &\es{\ll} \Im(\tau)^{-1/2}
	\frac{|\jth(\tau)|^{3/8}|\jth(z)|^{5/8}\Im(z)^{-3/2}}{|\jth(z)-\jth(\tau)|},\label{eq:KplusJestimate}\\
	|K_{-}(\tau, z)| &\es{\ll} \Im(\tau)^{-1/2}
	\frac{|\jth(\tau)|^{7/8}|\jth(z)|^{1/8}\Im(z)^{-3/2}}{|\jth(z)-\jth(\tau)|}.\notag
	\end{align}
From now on, we make all estimates for $z\in \ell$ with $\Im(z)<1/2$,
and we define $t>0$ in such a way that $\jth(z)=1/64+it$.
For such~$z$ we can use the following simple geometric estimate
(recall that $\Im(\jth(\tau))<0$)
	\begin{equation}\label{eq:Jestimate}|\jth(\tau)-\jth(z)| \es{\gg} \sqrt{|\jth(\tau)|^2+|\jth(z)|^2}.\end{equation}
Let $w\colon\CC\sm[0,\frac{1}{64})\to\mathcal{D}$ be the inverse
of~$\jth$ on $\mathcal{D}$, so that $z = w(1/64+it)$.
We have $\jth'(\tau) = i\pi f(\tau)\jth(\tau)$,
where $f(\tau) = \theta^4(\tau)(1-2\lambda(\tau))$ is a holomorphic
modular form of weight~$2$. Since $f$ does not vanish at the
cusp $1$, we have that $|f(z)| \gg \Im(z) ^{-2}$, and thus
	\begin{equation}\label{eq:dz}
	|\dz| \= |w'(1/64+it)|\,|\dt|
	\= \frac{|\dt|}{|\jth'(w(\frac{1}{64}+it))|}
	\es{\ll} \frac{|\dt|}{|\jth(z)|\cdot\Im(z)^{-2}}.
	\end{equation}
Note that this last estimate readily implies that $\ell$ has
finite length.

From inequality \eqref{eq:kernel_estimate1} it follows that it
is enough to find a bound for
	\[\int_\ell |K_\eps(\tau,z)|\,\Im(z)^{-m}\,|dz|\quad\mbox{ for } m\geq 0.\]
From inequalities \eqref{eq:KplusJestimate}, \eqref{eq:Jestimate}, \eqref{eq:dz} we deduce
	\[\int_\ell |K_{+}(\tau,z)|\,\Im(z)^{-m}\,|dz|\ll  \int_{0}^{\infty}\frac{|\jth(\tau)|^{3/8}t^{-3/8}\Im(z)^{1/2-m}}
	{\Im(\tau)^{1/2}\sqrt{t^2+|\jth(\tau)|^2}}\dt.\]
We will also need the estimate $|\jth(z)| \es{\gg} e^{\pi/\Im(z)}$ for $\Im(z)$ small enough. Indeed, this inequality follows from the $q$-expansion~\eqref{eq:jthetacusp} of $\jth(z)$ at the cusp $1$.
This implies that $\Im(z)^{-m}\ll_m\log^m(1+|\jth(z)|)$.
Thus, we have
	\begin{align*}
	 \int_\ell |K_{+}(\tau,z)|\,\Im(z)^{-m}\,|dz| & \ll \Im(\tau)^{-1/2}\int_{0}^{\infty}\frac{|\jth(\tau)|^{3/8}
	t^{-3/8}\log^m(1+t)\dt}{\sqrt{|\jth(\tau)|^2+t^2}}\\
	& \=
	\Im(\tau)^{-1/2}\int_{0}^{\infty}\frac{t^{-3/8}\log^m(1+t|\jth(\tau)|)\dt}{\sqrt{1+t^2}}.
	\end{align*}
By using an obvious inequality $\log(1+ab)\le \log(1+a)+\log(1+b)$,
we estimate the last integral by
	\begin{equation*}
	\Im(\tau)^{-1/2}\sum_{j=0}^{m}\binom{m}{j}\log^{j}(1+|\jth(\tau)|)	\int_{0}^{\infty}\frac{t^{-3/8}\log^{m-j}(1+t)\dt}{\sqrt{1+t^2}}
	\es{\ll} \sum_{j=0}^{m}c_{j,m}\Im(\tau)^{-j-1/2},
	\end{equation*}
where $c_{j,m}=\binom{m}{j}\int_{0}^{\infty}(1+t^2)^{-1/2}t^{-3/8}\log ^{m-j}(1+t)\dt$ are finite constants, and we have used the
inequality $\log(1+|\jth(\tau)|) \ll \Im(\tau)^{-1}$ that
follows from~\eqref{eq:jthetacusp}.
	
The estimates in the case ``$\eps=-$'' are completely analogous,
except that we need to change the exponent $3/8$ to $7/8$.
\end{proof}

We are now ready to prove Theorem~\ref{thm:mainest}.
\begin{proof}[Proof of Theorem~\ref{thm:mainest}]
	Let $\tau \in \HH$ be an arbitrary point in the upper half-plane with $\Im(\tau)\le1$
	that does not lie on the boundary of the fundamental domain $\mathcal{D}$
	or any of its translates by elements of~$\Gamma_{\theta}$.
	Let $\tau=\frac{a\tau'+b}{c\tau'+d}$,
	where $\tau'\in\mathcal{D}$ and $A=(\smat abcd) \in \Gamma_{\theta}$.
	By~\eqref{eq:funceq} we have
	\[
	\chi_{\eps}(A)j_{\theta}(\tau', A)F_{\eps}\Big(\frac{a\tau'+b}{c\tau'+d}\Big)
	\= F_{\eps}(\tau') - \phi_{A}(\tau').
	\]
	Combining the results of Lemma~\ref{lem:kernel_estimate} and
	Lemma~\ref{lem:cocyclelemma} (which we apply to $\psi_A=\phi_A$ with~$\alpha=0$ and~$\beta=1/2$) we obtain
	\begin{multline*}
		|F_{\eps}(\tau)|
		\leqs \frac{\Im(\tau')^{1/4}}{\Im(\tau)^{1/4}}|F_{\eps}(\tau')| + \frac{\Im(\tau')^{1/4}}{\Im(\tau)^{1/4}}|\phi_{A}(\tau')| \\
		\leqs C_0\frac{\Im(\tau')^{1/4} + \Im(\tau')^{-1/4}}{\Im(\tau)^{1/4}}
		+ \Im(\tau')^{1/4}(1 + \Im(\tau)^{-5/4}+2\,\Im(\tau)^{-7/4}).
	\end{multline*}
	(Here $C_0$ is the constant from Lemma~\ref{lem:kernel_estimate}.)
	If $c=0$, then $\Im(\tau')=\Im(\tau)$ and thus
		\[
		|F_{\eps}(\tau)| \leqs C_0(1+\Im(\tau)^{-1/2}) +\Im(\tau)^{1/4}+\Im(\tau)^{-1}+2\,\Im(\tau)^{-3/2}.
		\]
	If, on the other hand, $c \ne 0$, then we have $\Im(\tau) < \Im(\tau')$ and
		\[\Im(\tau)\Im(\tau') \= \frac{\Im(\tau')^2}{|c\tau'+d|^2} \le 1,\]
	and we get the estimate
		\[
		|F_{\eps}(\tau)| \leqs 2C_0\Im(\tau)^{-1/2} + \Im(\tau)^{-1/4}+\Im(\tau)^{-3/2}+2\,\Im(\tau)^{-2}.
		\]
	Therefore, an application of Lemma~\ref{lem:heckebound} gives
		\[|b_n^{\eps}(x)| \ll n^{2}.\]
\end{proof}
The exponent ``$2$'' in Theorem~\ref{thm:mainest} is not optimal,
but for the proof of Theorem~\ref{thm:summation} any polynomial bound would suffice.
\smallskip

\section{Proof of the main results}
\label{sec:funcan}
Now that we know that $b_n^{\eps}(x)$ have polynomial growth in $n$,
the proof of Theorem~\ref{thm:summation} and Theorem~\ref{thm:isomorphism} is not hard.

Recall the definition of Schwartz functions:
	\[\mathcal{S} \= \{f\in C^{\infty}(\RR) \colon \|f\|_{\alpha,\beta}
    < \infty\; \forall \alpha,\beta\ge 0\},\]
where the seminorms $\|\cdot\|_{\alpha,\beta}$ are defined by
	\[\|f\|_{\alpha,\beta}\=\sup_{x\in\RR}|x^{\alpha}f^{(\beta)}(x)|.\]
Convergence in $\mathcal{S}$ is defined in terms of this family of
seminorms, i.e., $f_n\to f$ if and only if
$\|f_n-f\|_{\alpha,\beta}\to 0$ for all $\alpha,\beta\ge0$.

\begin{proof}[Proof of Theorem~\ref{thm:summation}]
Let $\mathcal{S}_{even}$ be the space of even Schwartz functions.
Let us define
	\[a_n(x)\es{:=}\frac{b^+_n(x)+b^-_n(x)}{2}.\]
Lemma~\ref{prop:schwartz_eigenfunction} implies that
	\[\widehat{a}_n(x)\=\frac{b^+_n(x)-b^-_n(x)}{2}.\]
Our aim is to show that \eqref{eq:summation} holds for all $f\in \mathcal{S}_{even}$. Theorem~\ref{thm:mainest} implies that the series
	\[\sum_{n=0}^{\infty} a_{n}(x)f(\sqrt{n})+\sum_{n=0}^{\infty} \widehat{a}_{n}(x)\widehat{f}(\sqrt{n})\]
converges absolutely.
Moreover, it follows from the definition of $b_n^{\eps}$ and the
functional equations~\eqref{eq:genfun_main} that for any $\tau\in\HH$ we have
\begin{equation}\label{eq:summation_for_gaussians}
	e_{\tau}(x) \= \sum_{n=0}^{\infty} a_{n}(x)\,e_{\tau}(\sqrt{n})+\sum_{n=0}^{\infty} \widehat{a}_{n}(x)\,\widehat{e}_{\tau}(\sqrt{n}),
\end{equation}
where $e_{\tau}(x)=e^{i\pi\tau x^2}$.

For $x\geq 0$ consider the linear functional $\phi_x$ on $\mathcal{S}_{even}$ given by
	\[\phi_x(f)\es{:=}
	f(x)-\sum_{n=0}^{\infty} a_{n}(x)f(\sqrt{n})-\sum_{n=0}^{\infty} \widehat{a}_{n}(x)\widehat{f}(\sqrt{n}).\]
It follows from~Theorem~\ref{thm:mainest} that $\phi_x$ is a tempered distribution,
i.e., it is continuous with respect to convergence in $\mathcal{S}_{even}$.
From equation~\eqref{eq:summation_for_gaussians} we see that~$\phi_x$ vanishes on the
subspace spanned by $\{e_\tau\}_{\tau\in\HH}$. Our goal is to show that~$\phi_x$
vanishes on the whole $\mathcal{S}_{even}$.

Let $\mathcal{C}$ be the space of compactly supported even $C^\infty$ functions
on $\RR$.
Recall, that $\mathcal{C}$ dense in $\mathcal{S}_{even}$ (see~\cite{Vl}*{pp.~74-75}).
Therefore, it suffices to show~\eqref{eq:summation} for $f\in\mathcal{C}$. Let $f$ be a function in $\mathcal{C}$. We may assume that
\[f(x)\=F(x^2)\,e^{-\pi x^2}\]
where $F$ is a $C^\infty$ function with compact support on $\RR$. Consider the one-dimensional Fourier transform of $F$
	\[\ftt{F}(s):=\int_{-\infty}^{\infty}F(t)\,e^{-2\pi i s t}\,\dt.\]
Note, that $\widehat{F}$ is a Schwartz function. By the Fourier
inversion formula we have
	\[f(x)\=F(x^2)\,e^{-\pi x^2}
	\=\int_{-\infty}^{\infty}\ftt{F}(s)\, e^{2\pi i s x^2-\pi x^2}\,ds
	\= \int_{-\infty}^{\infty}\ftt{F}(s)\,e_{i+2s}(x)\,ds.\]
Define
    \[h_T:=\int_{-T}^{T}\ftt{F}(s)\,e_{i+2s}(x)\,ds.\]
It is easy to see that for all seminorms $\|\cdot\|_{\alpha,\beta}$
	\[\|f-h_T\|_{\alpha,\beta}\to 0\quad\mbox{as}\;\,T\to\infty. \]
Therefore, for all $x\geq0$
	\[\phi_x(f-h_T)\to 0\quad\mbox{as}\;\,T\to\infty. \]
On the other hand, we have
    \[\phi_x(h_T)=\int_{-T}^{T}\ftt{F}(s)\,\phi_x(e_{i+2s})
    \,ds=0.\]
This finishes the proof of Theorem~\ref{thm:summation}.\end{proof}

We are also ready to prove Theorem~\ref{thm:isomorphism}.
\begin{proof}[Proof of Theorem~\ref{thm:isomorphism}] First, we observe that
the image of $\Psi$ is contained in the kernel of~$L$. Indeed, the Poisson
summation formula implies
	\[\sum_{n\in\ZZ}f(n)\=\sum_{n\in\ZZ}\widehat{f}(n)\]
for all $f\in\mathcal{S}$ as well as $f\in\mathcal{S}_{even}$. This identity is equivalent to $L\circ\Psi(f)=0$.

Next, we construct the function $\Phi:\ker L\to \mathcal{S}_{even}$ such that $\Psi\circ\Phi=\mathbb{I}_{\ker L}$. To this end we consider the map
	\[\Phi:\ker L\to \mathcal{S}_{even},\qquad ((x_n),(y_n))\mapsto \sum_n x_n\,a_n(x)+y_n\,\widehat{a}_n(x).\]
We need to show that $\Phi$ is well-defined.
Since $\mathcal{S}$ is complete with respect to the family of norms $\|\cdot\|_{\alpha,\beta}$ it is enough to prove that for any fixed~$\alpha, \beta\ge0$ the sequences $(\|a_n\|_{\alpha,\beta})_n$
and $(\|\ftt{a_n}\|_{\alpha,\beta})_n$ have at most polynomial growth
in $n$. Equivalently, it is enough to prove that the
sequences $(\|b_n^{\eps}\|_{\alpha,\beta})_n$ have polynomial growth.
	
As before, let $Q_k(x,z)$ be the polynomial defined
by~\eqref{eq:polyqk}. Let $U(\tau,x)$ be the generating function
	\[U(\tau,x) \= x^{\alpha}\frac{d^{\beta}}{dx^{\beta}}F_{\eps}(\tau,x)
	\= x^{\alpha}\,\sum_{n=0}^{\infty}
	\frac{d^{\beta}}{dx^{\beta}}b_n^{\eps}(x)\,e^{i\pi n \tau}.\]
Then, following the proof of Proposition~\ref{prop:genfunceq}, we see that
the generating function $U$ satisfies the functional equation
	\[U(\tau) - (U|_{1/2}^{-\eps}A)(\tau) \= \phi_{A}(\tau),\]
where $\phi_A$ is the cocycle defined by
	\begin{equation*}
		\begin{split}
			\phi_{T^2}(\tau) &\= 0,\\
			\phi_{S}(\tau) &\=
			x^{\alpha} Q_{\beta}(x,\tau)e^{i\pi x^2\tau} + \eps(-i\tau)^{-1/2}x^{\alpha}Q_{\beta}(x,-1/\tau)e^{i\pi x^2(-1/\tau)}\,.
		\end{split}
	\end{equation*}
Using the estimates
	\[|x^k\tau^l e^{i\pi x^2\tau}|\es{\ll} |\tau|^{l} \Im(\tau)^{-k/2}
	\es{<} \Im(\tau)^{-k}+|\tau|^{2l}\]
and
	\[|x^k\tau^{-l} e^{i\pi x^2(-1/\tau)}|
	\es{\ll} |\tau|^{k-l} \Im(\tau)^{-k/2}
	\es{<} \Im(\tau)^{-k}+|\tau|^{2k-2l},\]
and in case $k<l$ the replacing $|\tau|^{2k-2l}$ by $\Im(\tau)^{2k-2l}$,
we see that Lemma~\ref{lem:cocyclelemma} can be applied
to $\{\phi_A\}_{A\in\Gamma_{\theta}}$
(for some choice of $\alpha$ and $\beta$ in Lemma~\ref{lem:cocyclelemma}).
Lemma~\ref{lem:kernel_estimate} implies that for $\tau\in\mathcal{D}$
we have
	\[U(\tau,x) \es{\ll} 1+\Im(\tau)^{-(\alpha+\beta+1)/2}.\]
Arguing the same way as in the proof of Theorem~\ref{thm:mainest}
we obtain that for some $C>0$ and all~$\tau\in\HH$ with $\Im(\tau)<1$
we have $|U(\tau,x)| \ll \Im(\tau)^{-C}$,
which implies that $\|b_{n}^{\eps}\|_{\alpha,\beta} \ll n^C$. Therefore, the map $\Phi$ is well-defined.

Now Theorem~\ref{thm:summation} implies that $\Phi\circ\Psi=\mathbb{I}_{\mathcal{S}_{even}}$ and Proposition~\ref{prop:schwartz_eigenfunction} implies that $\Psi\circ\Phi=\mathbb{I}_{\ker L}$. This finishes the proof.
\end{proof}

\section{Interpolation basis for odd functions}
\label{sec:odd}
The case of odd Schwartz functions is very similar to the even case.
The proofs are easy enough to adapt to this case,
so we will just give the general outline.
The role of the Gaussian $e_{\tau}(x)=e^{i\pi \tau x^2}$ is played by
the Schwartz function
	\[o_{\tau}(x) \= xe^{i\pi \tau x^2},\]
that satisfies
	\[\ftt{o_{\tau}}(\xi) \= -i(-i\tau)^{-3/2} o_{-1/\tau}(\xi).\]
To construct the interpolation basis for odd Schwartz functions we use
the same idea as before: to get an eigenfunction we
integrate $o_{\tau}$ over $\tau$ with some ``modular weight''.
More precisely, let $h_n^{\eps}\colon\HH\to\CC$
be holomorphic functions with the following properties:
\begin{align*}
	\begin{split}
		h_n^{\eps}(z+2) &\= h_n^{\eps}(z),\\
		(-iz)^{-1/2}h_n^{\eps}(-1/z) &\= \eps h_n^{\eps}(z),\\
		h_n^{+}(z) &\= q^{-n/2}+O(q^{1/2}),\; z\to i\infty,\\
		h_n^{-}(z) &\= q^{-n/2}+O(1),\; z\to i\infty,\\
		h_n^{\eps}(1+i/t) &\es{\to} 0,\; t\to\infty.
	\end{split}
\end{align*}
Once again, we may assume that they are of the form
\begin{align} \label{eq:ansatz_odd}
	\begin{split}
		h_n^{+}(z)& \= \theta(z)Q_n^{+}(\jth^{-1}(z)),\\
		h_n^{-}(z)& \= \theta(z)(1-2\lambda(z))Q_n^{-}(\jth^{-1}(z)),
	\end{split}
\end{align}
where $Q_n^{\pm}\in\QQ[x]$ are monic of degree $n$ and $Q_n^{-}$ has no constant term. The first few of these functions are	
	\begin{align*}
	\begin{split}
	h_0^{+}& \= \theta,\\
	h_1^{+}& \= \theta\cdot(\jth^{-1}-26),\\
	h_2^{+}& \= \theta\cdot(\jth^{-2}-50\jth^{-1}+76),
	\end{split}
	\quad
	\begin{split}
	h_1^{-}& \= \theta\cdot(1-2\lambda)\cdot(\jth^{-1}),\\
	h_2^{-}& \= \theta\cdot(1-2\lambda)\cdot(\jth^{-2}-18\jth^{-1}),\\
	h_3^{-}& \= \theta\cdot(1-2\lambda)\cdot(\jth^{-3}-42\jth^{-2}+168\jth^{-1}).
	\end{split}
	\end{align*}
By the same arguments as in the even case, we establish generating functions for $h_n^{\eps}$,
which turn out to be the same, except for switching the roles of $\tau$
and $z$.
\begin{theorem}
	The generating functions for $\{h_n^{+}(z)\}_{n\ge0}$
	and $\{h_n^{-}(z)\}_{n\ge1}$ are given by
	\begin{align} \label{eq:kernels_odd}
		\begin{split}
			\sum_{n=0}^{\infty} h_n^{+}(z)e^{i\pi n\tau} \= \frac{\theta^3(\tau)(1-2\lambda(\tau))\theta(z)\jth(z)}{\jth(z)-\jth(\tau)}
			\= -K_{-}(z, \tau),\\
			\sum_{n=1}^{\infty} h_n^{-}(z)e^{i\pi n\tau} \=
			\frac{\theta^3(\tau)\jth(\tau)\theta(z)(1-2\lambda(z))}{\jth(z)-\jth(\tau)} \= -K_{+}(z, \tau).
		\end{split}
	\end{align}
\end{theorem}

Similarly to the even case, define $d_m^{\eps}\colon\RR\to\RR$ by
\[d_m^{\eps}(x) \= \frac12\int_{-1}^{1} h_m^{\eps}(z)\,xe^{i\pi x^2 z}\dz.\]
\begin{proposition} \label{prop:schwartz_eigenfunction_odd}
	The function $d_m^{\eps}\colon\RR\to\RR$ is odd, belongs to the Schwartz class, and satisfies
	\[\ftt{d_m^{\eps}}(x) \= (-i\eps)\,d_m^{\eps}(x)\]
	and
	\[d_m^{\eps}(\sqrt{n}) \= \delta_{n,m}\sqrt{n},\quad n\ge 1,\]
	where $\delta_{n,m}$ is the Kronecker delta. Moreover,
	\[\lim_{x\to 0}\frac{d_m^{+}(x)}{x} \= \delta_{m,0}.\]
\end{proposition}
Furthermore, we have the following estimate on the growth of $d_{n}^{\pm}(x)$
as a function of~$n$.
\begin{theorem} \label{thm:mainest_odd}
	For any $\eps\in\{+,-\}$ the
	numbers $d_n^{\eps}(x)$ satisfy
	\[d_n^{\eps}(x) \= O(n^{5/2})\]
	uniformly in $x$.
\end{theorem}
The proof of this estimate is also based on estimating the growth
for $\Im(\tau)\to0$ of the generating function
\begin{equation*} \label{eq:genfunseries_odd}
	G_{\eps}(\tau, x) \= \sum_{n\ge 0} d_n^{\eps}(x)e^{i\pi n \tau}.
\end{equation*}
The functional equations for $G_{\eps}$ are
	\begin{equation}\label{eq:genfun_main_odd}
	\begin{split}
	G_{\eps}(\tau,x) - G_{\eps}(\tau+2,x) &\= 0,\\
	G_{\eps}(\tau,x) + \eps(-i\tau)^{-3/2}G_{\eps}\Big(-\frac{1}{\tau},x\Big) &\=
	xe^{i\pi \tau x^2} + \eps(-i\tau)^{-3/2}xe^{i\pi (-1/\tau) x^2}.
	\end{split}
	\end{equation}
The difference in exponents of $(-i\tau)$ come from the fact that the
weight of $K_{\eps}(z,\tau)$ in variable~$\tau$ is now~$3/2$, but
with appropriate changes the proof still goes through.
Finally, we get the following interpolation theorem for odd Schwartz functions.
\begin{theorem}\label{thm:summation_odd}
	For any odd Schwartz function $f\colon\RR\to\RR$ and any $x\in\RR$ we have
	\begin{equation} \label{eq:summation_odd}
		f(x) \= d_0^{+}(x)\frac{f'(0)+i\ftt{f}\,'(0)}{2}\es{+}
		\sum_{n=1}^{\infty} c_{n}(x)\,\frac{f(\sqrt{n})}{\sqrt{n}}\es{-}
		\sum_{n=1}^{\infty} \ftt{c_{n}}(x)\,\frac{\ftt{f}(\sqrt{n})}{\sqrt{n}},
	\end{equation}
	where $c_n(x) = (d_n^{+}(x)+d_n^{-}(x))/2$.
\end{theorem}
As in the even case, the functional equations for $G_{\eps}$ show
that~\eqref{eq:summation_odd} holds for $o_{\tau}(x)$, so one only needs
to show that $o_{\tau}$ are dense in the space of odd Schwartz
functions, which can be done by an approximation argument, similarly to
the proof of Theorem~\ref{thm:summation}.

Let us also note that the even interpolation basis $\{a_n(x)\}_{n}$ is
defined using the kernel $K(\tau,z):=K_+(\tau,z)+K_{-}(\tau,z)$, and the
odd interpolation basis $\{c_n(x)\}_{n}$ is defined using the kernel $\widetilde K(\tau,z):=-K(z,\tau)$.
Thus, even though we have dealt with even and odd interpolation problems separately, \
there is a nice duality between the two.

\textbf{Remark.} As in the even case, using the explicit formula
for the kernels, we get
	\begin{align*}
	d_m^{+}\,'(0) &\= \delta_{m,0},\\
	d_m^{-}\,'(0) &\= -r_3(m),\quad m\ge 1,
	\end{align*}
where $r_3(m)$ is the number of representations of $m$ as the sum of squares
of $3$ integers. Taking $x=0$ in~\eqref{eq:summation_odd} we get the
following identity
	\[f'(0) + \sum_{n=1}^{\infty}\frac{r_3(n)f(\sqrt{n})}{\sqrt{n}}
	\= i\ftt{f}\,'(0)
	+\sum_{n=1}^{\infty}\frac{r_3(n)i\ftt{f}(\sqrt{n})}{\sqrt{n}},\]
valid for arbitrary odd Schwartz functions. As was pointed out to us
by Yves Meyer, this formula was previously found by
Guinand~\cite[p.~265]{Guin}.

\section{Open questions and concluding remarks}
Let us indicate some further directions and observations related
to Theorem~\ref{thm:summation}.
\subsection*{Function space.}
In this paper we have only worked with the space of Schwartz
functions, but it is interesting to ask in what generality the
interpolation formula~\eqref{eq:summation} holds. The best possible
scenario would be a positive answer to the following question.
    \begin{question}
	Do the results of Theorems~\ref{thm:summation}
    and~\ref{thm:summation_odd} hold whenever the sum on
    the right-hand side is well-defined and converges
    absolutely?
    \end{question}
Even to find explicit conditions for when the convergence
is absolute, one would need to obtain exact bounds on the
growth of $b_n^{\eps}(x)$, which appears to be difficult.
Let us outline a simple approximation argument that shows
that the interpolation formula is true whenever both $f$
and $\ftt{f}$ decay sufficiently fast:
\begin{proposition}
	Let $f$ be an even integrable function.
	If $f(x)$ and $\ftt{f}(x)$ are both bounded by~$(1+|x|)^{-13}$,
	then the summation formula~\eqref{eq:summation} holds.
\end{proposition}
\begin{proof}[Proof sketch]
Indeed, for every~$T>0$
consider the following linear operator $\mathcal{R}_{T}$ that takes
values in~$\mathcal{S}$:
	\[
	\mathcal{R}_{T}(f)(x) \=
	T^{1/2}e_{i/T}\cdot (e_{iT}\ast f)(x)
	\= T^{1/2}e^{-\pi x^2/T}\int_{-\infty}^{\infty}f(x-y)e^{-\pi T y^2}\dy.
	\]
The Fourier transform is then given by
	\[\ftt{\mathcal{R}_{T}(f)}(x)
	\= T^{1/2}e_{iT}\ast (e_{i/T}\cdot \ftt{f})(x)
	\= T^{1/2}\int_{-\infty}^{\infty}f(x-y)e^{-\pi T y^2-\pi(x-y)^2/T}\dy.\]
Then a routine calculation shows that
	\[|\mathcal{R}_{T}(f)(x)-f(x)| \es{\ll}
	(1-e^{-\pi x^2/T})|f(x)| + T^{-1/2}\max_{y\in[x-1,x+1]}|f'(y)|,\]
and similarly
	\[|\ftt{\mathcal{R}_{T}(f)}(x)-\ftt{f}(x)| \es{\ll}
	\Big(1-\frac{e^{-\pi x^2T/(1+T^2)}}{\sqrt{1+T^{-2}}}\Big)|\ftt{f}(x)| + T^{-1/2}\max_{y\in[x-1,x+1]}|\ftt{f}'(y)|.\]
By summing up these estimates for $x=\sqrt{n}$ over $n\ge1$ and
taking the limit as $T\to\infty$ we see that the proof will be complete
if we can show that $f'(x)$ and $\ftt{f}'(x)$ decay as $(1+|x|)^{-l}$ for
some $l>6$ (since $a_n(x) = O(n^{2})$).
We consider only bounding $f'(x)$, since one can obtain the other estimate
by interchanging~$f$ and~$\ftt{f}$.
It was pointed out to the authors by Emanuel Carneiro that
this can be done using the following simple observation:
if~$g$ is a $C^2$-smooth function on $[1,\infty)$
that satisfies $|g(x)| \ll x^{-k}$ and $|g''(x)| \ll 1$
then~$|g'(x)|\ll x^{-k/2}$. Indeed, then by the Fourier inversion
formula we have $|f''(x)| \ll 1$, so we can apply the observation to get
$|f'(x)| \ll (1+|x|)^{-13/2}$, and thus we are done.

To prove the above observation: let $|g''(x)| \le 1$ and $|g(x)| \le Cx^{-k}$.
Then Taylor's theorem with remainder in the Lagrange form implies that
for any $\Delta\ge 0$ we have
	\[|g(x+\Delta)-g(x)-g'(x)\Delta| \leqs \frac{\Delta^2}{2},\]
from which we get, taking $\Delta=2\sqrt{Cx^{-k}}$, that
	\[|g'(x)| \leqs \frac{\Delta}{2}+\frac{2Cx^{-k}}{\Delta}
	\= 2\sqrt{C}\,x^{-k/2},\]
as required.
\end{proof}
Note that the number ``13'' in the above proposition can be
improved by using more careful estimates.

\subsection*{Relation to the Laplace transform.}
The basis functions that we have constructed are all of the shape
	\[f(x) \= \frac12\int_{-1}^{1}g(z)e^{i\pi x^2z}\dz\]
for some weakly holomorphic modular form $g$
(in the odd case, $f$ is multiplied by $x$).
To get an alternative expression for $f$ we can shift the contour
of integration to the rectangular line passing
through $-1$, $-1+iT$, $1+iT$, and $1$. A simple computation
then shows that
	\[f(x) \= \sin(\pi x^2)\int_{0}^{T}g(1+it)e^{-\pi x^2t}\dt
	+ e^{-\pi x^2 T}\int_{-1}^{1}g(s+iT)e^{i\pi s x^2}ds.\]
If we take $T$ to infinity, then we see that for all $x^2$ greater than
the order of the pole of~$g$ at~$i\infty$ we have
	\[f(x) \= \sin(\pi x^2)\int_{0}^{\infty}g(1+it)e^{-\pi x^2t}\dt.\]
The integral on the right is simply the Laplace transform of $g(1+it)$
evaluated at $\pi x^2$. This can be used to show that all but finitely
many real zeros of $b_{m}^{\pm}(x)$ are of the form~$\pm\sqrt{n}$.
Combined with the $q$-expansion of~$g(1+z)$ at infinity, this also
implies that $b_{m}^{\pm}$ extends analytically to an entire function.
Alternatively, this also follows directly from the
definition~\eqref{eq:bm_definition}.

\subsection*{Sine-sinh ratio.}
The function~$d_0^{+}(x)$ is quite special. Recall that it is defined by
	\[d_0^{+}(x) \= \frac12\int_{-1}^{1}\theta(z)\,xe^{i\pi x^2z}\dz.\]
Changing the contour of integration as before, we get
	\[d_0^{+}(x) \= x\sin(\pi x^2) \int_{0}^{\infty}\theta(1+it)e^{-\pi x^2 t}\dt.\]
Next, integrating the $q$-expansion of $\theta$ termwise and using the identity
	\[\sum_{n\in\ZZ}\frac{(-1)^n}{\pi(x^2+n^2)} \= \frac{1}{x\sinh(\pi x)}\]
we find that $d_0^{+}(x)$ is, in fact, an elementary function:
	\[d_0^{+}(x) \= \frac{\sin(\pi x^2)}{\sinh(\pi x)}.\]
Note that $d_0^{+}(x)$ and its Fourier transform $\ftt{d_0^{+}}(x) = (-i)d_0^{+}(x)$
both vanish at $x=\pm \sqrt{n}$ for all $n\ge 0$. It
follows from Theorems~\ref{thm:summation} and~\ref{thm:summation_odd} that any Schwartz function with this property is of the
form $\alpha d_0^+$.

It appears that this function was first considered by Ramanujan in~\cite{Ram},
where he studies a number of integrals involving similar expressions,
and, in particular, shows the Fourier invariance of $d_0^{+}$
(see~\cite{Ram}*{eq.~34}).
It is also directly related to the so-called Mordell integral~\cite{Mo}, which
played an important role in Zwegers's seminal work on mock theta functions~\cite{Zw}.

\end{document}